\documentclass[12pt,reqno]{amsart}
\usepackage[english]{babel}

\usepackage{geometry}
\usepackage{enumerate}
\usepackage{url}
\usepackage[hidelinks]{hyperref}
\usepackage[sort,compress]{cite}
\usepackage{amsmath}
\usepackage{amsthm}
\usepackage{amssymb}
\usepackage{bbm}
\usepackage{array}

\usepackage{graphicx}
\usepackage{xcolor}

\newtheorem{lemma}{Lemma}[section]

\newtheorem{proposition}[lemma]{Proposition}
\newtheorem{theorem}[lemma]{Theorem}

\newtheorem{corollary}[lemma]{Corollary}

\newtheorem{setting}[lemma]{Setting}

\providecommand{\N}{{\ensuremath{\mathbbm{N}}}}
\providecommand{\Z}{{\ensuremath{\mathbbm{Z}}}}
\providecommand{\R}{{\ensuremath{\mathbbm{R}}}}

\renewcommand{\P}{{\ensuremath{\mathbbm{P}}}}
\providecommand{\E}{{\ensuremath{\mathbbm{E}}}}

\providecommand{\1}{{\ensuremath{\mathbbm{1}}}}

\providecommand{\cov}{\mathrm{cov}}

\title{Random walks with drift in the positive quadrant}
\author{Tuan Anh Nguyen}
\address{Faculty of Mathematics\\Bielefeld University}
\email{tnguyen@math.uni-bielefeld.de}
\author{Vitali Wachtel}
\address{Faculty of Mathematics\\Bielefeld University}
\email{wachtel@math.uni-bielefeld.de}

\allowdisplaybreaks
\begin{document}
\maketitle
\begin{abstract}
 We consider two dimensional random walks conditioned to stay in the positive quadrant. Assuming that the increments of the walk have finite second moments and that the drift vector is co-oriented with one of two axes, we construct positive harmonic functions for such walks and find tail asymptotics for the exit time from the positive quadrant. Moreover, we prove integral and local limit theorems. Finally, we apply our local limit theorems to singular lattice walks with steps $\{(1,-1),(1,1),(-1,1)\}$ and determine asymptotics for the number of walks of length $n$ which end on the line $\{(k,1),\ k\ge1\}$. 
\end{abstract}
\tableofcontents

\section{Introduction}
This paper is devoted to the asymptotic behaviour of two-dimensional random walks conditioned to stay in the positive quadrant. In contrast to the already existing literature, we shall concentrate on the case when the drift of the walk is non-zero and is co-oriented with one of the axes.

The study of multidimensional random walks in cones has become quite popular in the recent past. This interest is reasoned by numerous connections with models in stochastics and combinatorics.
Examples are: non-intersecting paths, branching processes in random environment and enumeration problems for spatially restricted lattice walks.

In the series of papers \cite{DW2015,DW2019,DW2024a}, Denisov and Wachtel have developed a method, which allows one to construct positive harmonic functions for random walks killed at leaving a cone in the case when the drift of the walk is zero. They have also found the tail asymptotics for the exit time and proved integral and local conditioned limit theorems.

The case when the drift belongs to the complement of the cone has been studied by Duraj~\cite{Duraj2014}. The walks with the drift showing inside the cone were studied in \cite{Garbit-Raschel-2023,HRT2023}. The authors have studied the probability that such walks stay in the cone for all times and have constructed bounded harmonic functions.

Thus, the only unstudied situation where walk leaves the cone with probability one is the case when the drift vector belongs to the boundary of the cone.

In the present paper we shall consider the special cone $K=\R_+^2$ only. The reason for the choice of the cone is twofold. First, our study was motivated by a question of Mireille Bousquet-Mélou on combinatorics of lattice paths in the positive quadrant. We shall give precise formulation her question and our answer later. Second, simplicity of the cone $\R_+^2$ allows us to describe the effect of the drift more transparently.

Let $\{X(n)\}$ be a sequence of i.i.d. random vectors with values in $\R^2$ such that 
\begin{equation}
\label{eq:1st.moment} 
\E[X_2(1)]=0
\quad\text{and}\quad
\E[X_1(1)]=:\mu_1>0.
\end{equation}
The mean of the vector $X(1)$ will be denoted by $\mu$, that is, $\mu=(\mu_1,0)$.
Furthermore, we shall assume that the second moments of the increments are finite and given by 
\begin{align}
\label{eq:2nd.moment}    
\Sigma=
    \begin{pmatrix}
        \cov(X_1(1),X_1(1)),&\cov(X_1(1),X_2(1))\\
        \cov(X_2(1),X_1(1)),&\cov(X_2(1),X_2(1))        
    \end{pmatrix}
    =
    \begin{pmatrix}
\sigma_1^2&\rho\\
\rho&\sigma_2^2
\end{pmatrix}
.
\end{align}
Let $\{S(n)\}$ denote the random walk with the increments
$\{X(n)\}$, that is,
$$
S(n)=X(1)+X(2)+\ldots+X(n),\quad n\ge1. 
$$
For every $x\in[0,\infty)^2$ we define the stopping time 
$$
T_x:=\inf\{n\ge1:x+S(n)\notin\R_+^2\}.
$$
This is, $T_x$ is the first time when the walk $\{x+S(n)\}$ exits the positive quadrant. Our main purpose is to determine tails asymptotics of this exit time and to prove limit theorems for $S(n)$ conditioned on $T_x>n$. To state our results we need to introduce some additional notations. Besides the exit times from the positive quadrant we define exit times from the upper half-plane:
$$
\tau_x:=\inf\{n\ge1:x+S(n)\notin\R\times\R_+\},
\quad x\in\R\times[0,\infty).
$$
Actually, $\tau_x$ is a stopping time with respect to the $1$-dimensional walk $\{S_2(n)\}$. Indeed,
$$
\tau_x=\inf\{n\ge1: x_2+S_2(n)\le0\}, 
\quad x\in\R\times[0,\infty).
$$
Let $\chi^{-}:=-S_2(\tau_0)$ be the first weak descending ladder height of the walk $\{S_2(n)\}$. Let $\{\chi^{-}_k\}$ be independent copies of the variable $\chi^-$ and define
\begin{align*}
V(u):=\1_{\{u\ge 0\}}+\sum_{k=1}^{\infty }\P(\chi _{1}^{-}+\ldots
+\chi _{k}^{-}\le u).  
\end{align*}
This renewal function is harmonic for $\{S_2(n)\}$ killed at the stopping time $\tau_x$, that is,
$$
V(x_2)=\E[V(x_2+S_2(1));\tau_x>1],\quad x_2>0
$$
and plays a crucial role in our construction of a positive harmonic function for $\{S(n)\}$ killed at leaving the positive quadrant.

Finally, we define exit times from the right half-plane:
$$
\sigma_x:=\inf\{x+S(n)\notin (0,\infty)\times\R\}
=\inf\{n\ge 1: x_1+S_1(n)\le0\},\quad x\in[0,\infty)\times\R.
$$
\begin{theorem}
\label{thm:tail_and_limit}
Assume that \eqref{eq:1st.moment} and \eqref{eq:2nd.moment} hold.
Then, for every fixed $x\in\R\times(0,\infty)$,
\begin{equation}
\label{eq:T_tail}
\P(T_x>n)\sim \varkappa W(x)n^{-1/2},
\end{equation}
where 
$$
\varkappa:=\sqrt{\frac{2}{\pi}}\E[\chi^-]
$$ 
and 
$$
W(x):=V(x_2)-\E[V(x_2+S_2(\sigma_x));\tau_x>\sigma_x,\sigma_x<\infty].
$$
This function is harmonic for $S(n)$ killed at leaving the quadrant,
that is,
\begin{align}
\label{eq:W-harm}
W(x)=\E[W(x+S(1));T_x>1], \quad x\in\R\times(0,\infty).
\end{align}
Furthermore, 
for every fixed $x\in\R_+^2$ and all $x\in\R_+^2$ one has
\begin{align}
\label{eq:limit}
\lim_{n\to\infty}
\sqrt{n}\P \!\left(\frac{x+S(n)-n\mu}{\sqrt{n}}\in u+\Delta;T_x>n\right)
= \varkappa W(x)\int_{u+\Delta}p(y)\,dy,
\end{align} 
where $\Delta= [0,1)^d$ and the function $p(y)$ is defined in \eqref{eq:p-def}.
\end{theorem}

We now turn to local limit theorems for conditioned random walks. We shall assume that the vectors $\{X(k)\}$ take values on the lattice $\Z^2$. Then there exist integers $d_1>a_1\ge0$ and $d_2>a_2\ge0$ such that
\begin{align}
\label{eq:loc-reduction}
X_1(k)=a_1+d_1Y_1(k)
\quad\text{and}\quad
X_2(k)=a_2+d_2Y_2(k),
\end{align}
where the vectors $\{Y(k)\}$ have aperiodic distribution on the lattice
$\Z^2$. Specializing Corollary 1 in Stone~\cite{Stone1967} to the lattice case
and noting that 
$$
\cov(Y_i(1),Y_j(1))
=\frac{1}{d_id_j}\cov(X_i(1),X_j(1)),
\quad i,j\in\{1,2\},
$$
we conclude that 
\begin{equation}
\label{eq:llt}
\sup_{z\in D_n}\left|n\P(S(n)=z)
-\frac{d_1d_2}{2\pi|\Sigma|^{1/2}}e^{-(x-\mu)\Sigma^{-1}(x-\mu)^{T}/2}
\right|\to0\quad\text{as }n\to\infty
\end{equation}
and
\begin{align}\label{eq:llt'}
\P(S(n)=z)=0,\quad z\notin D_n,
\end{align}
where 
$$
D_n:=\left\{z\in\Z^2:\frac{z_i-a_in}{d_i}\in\Z,\ i=1,2\right\}.
$$
\begin{theorem}\label{t11}
Assume that the vectors $X(k)$ take values on $\Z^2$ and
let $d_1,d_2$ be defined by \eqref{eq:loc-reduction}.
Then, for every fixed $x\in\N^2$,
\begin{align*}
\lim_{n\to\infty}\sup_{y\in D_n(x)}
\left\lvert
n^{3/2}\P(x+S(n)=y, T_x>n)-d_1d_2\varkappa W(x)p\!\left(\frac{y-n\mu}{\sqrt{n}}\right)
\right\rvert=0
\end{align*}
where 
\begin{align*}
D_n(x):=\{y\in\N^2:\ y-x\in D_n\}.
\end{align*}
Furthermore,
$$
\P(x+S(n)=y, T_x>n)=0,\quad x\notin D_n(x).
$$
\end{theorem}
We next turn to local asymptotics in the case when $y_2=o(\sqrt{n})$.
To formulate tour results we have to introduce a further renewal function. Define
$\tau^+:=\inf\{n\ge1:S_2(n)>0\}$ and 
$\chi^+:=S_2(\tau^+)$. Set
\begin{align}
H(u):=\1_{\{u>0\}}+\sum_{k=1}^{\infty }\P(\chi _{1}^{+}+\ldots
+\chi _{k}^{+}<u)  \label{r01},
\end{align}
where $\{\chi_k^+\}$ are independent copies of the variable $\chi^+$.
\begin{theorem}\label{t12}
Assume that the conditions of Theorem~\ref{t11} are valid. Uniformly in
$y\in D_n(x)$ such that $y_2=o(\sqrt{n})$ and $y_1>\frac{3}{4}\mu_1 n$ we have
\begin{align*}
\P(x+S(n)=y,T_x>n)=
d_1d_2q\left (\frac{y_1-n\mu_1}{\sqrt{n}}\right)\frac{H(y_2)W(x)}{n^2}
+o\left(\frac{H(y_2)}{n^{2}}\right).
\end{align*}An
explicit formula for the function $q$ is given in \eqref{eq:q-def} below. 
\end{theorem}
This theorem allows one to determine the asymptotic behaviour of
$\P(x_2+S_2(n)=y_2,T_x>n)$ for a fixed starting point $x$ and for fixed $y_2>0$.
\begin{corollary}\label{t15}
For every fixed $y_2$ one has
    \begin{align*}
        \P(x_2+S_2(n)=y_2,T_x>n)\sim d_1d_2\frac{W(x)H(y_2)}{n^{3/2}}\int_{-\infty}^\infty q(z)\,dz.
    \end{align*}
\end{corollary}

We now apply our results to the following question posed to us by Mireille Bousquet-Mélou. Consider lattice walks with possible steps
$(1,-1)$, $(1,1)$ and $(-1,1)$ which are confined to the positive quadrant $\N^2$. Then one wants to determine asymptotics of the number of walks of length $n$, which begin at $x\in\N^2$ and end at the line
$\{(k,1),\,k\ge1\}$. This question can be answered by applying Corollary~\ref{t15}. To do so, we first transfer the combinatorial problem into a probabilistic one, by considering increments  with the uniform distribution on the step set introduced above. Since the drift of this increment equals $\left(\frac{1}{3},\frac{1}{3}\right)$, we then perform an appropriate exponential change of measure which leads to a walk satisfying \eqref{eq:1st.moment}. Then, applying Corollary~\ref{t15} and performing the inverse change of measure we conclude that for every starting point $x$ there exists a constant $C(x)$ such that the number of walk ending at $\{(k,1),\,k\ge1\}$ is asymptotically equivalent to
$$
C(x)\frac{2^{3n/2}}{n^{3/2}}.
$$

The remaining part of the paper is organized as follows. In Section~\ref{s02} we consider two-dimensional Brownian motion killed at leaving a half-space. In Section~\ref{s03} we consider random walks in the upper half-plain. In Section~\ref{s04} we prove the tail asymptotic of $T_x$ and the conditional limit theorem, Theorem~\ref{thm:tail_and_limit}. In Section~\ref{s05} we prove the local limit theorem for the walk conditioned to stay in the positive quadrant, Theorem~\ref{t11}, as well as
Theorem~\ref{t12} and Corollary~\ref{t15}, and consider singular walks ending at one of the axes as an application.

\section{Two-dimensional Brownian motion killed at leaving a half-space}\label{s02}
In this paragraph we derive some explicit formulas for the transition density of a Brownian motion with drift killed at leaving a half-plane. We are particularly interested in the case when the drift is parallel to the boundary of a half-plane, since the corresponding conditional density, which is derived in Corollary~\ref{c10}, describes the limiting distribution in our main results.

Let $\mu\in\R^2 $ and let $\Sigma\in \R^{2\times 2}$ be positive definite, symmetric matrix. Let $B(t)$ be a Brownian motion with drift $\mu$ and with the covariance matrix $\Sigma$, i.e.,
\begin{align*}
\mu= 
\begin{pmatrix}
    \E [B_1(t)]\\
    \E [B_2(t)]
\end{pmatrix}
,\quad \Sigma=
\begin{pmatrix}
    \mathrm{cov}(B_1(1),B_1(1)),
    &\mathrm{cov}(B_1(1),B_2(1))\\
    \mathrm{cov}(B_2(1),B_1(1)),
    &\mathrm{cov}(B_2(1),B_2(1))
\end{pmatrix}.
\end{align*} We first give an explicit formula for the transition density of $B(t)$ killed at the stopping time
$$
\tau^{\mathrm{bm}}_x
:=\inf\{t\in [0,\infty)\colon B(t)\notin \R\times (0,\infty)\}.
$$
Let $K_t^\mathrm{bm}(x,\cdot )$ denote the density of the measure $\P(x+B(t)\in\cdot,\tau^\mathrm{bm}_x>t)$.
\begin{lemma}[cf.\ {\cite[Lemma 21]{HRT2023}}]
\label{c01}
For all $x,y\in \R\times (0,\infty)$ we have
\begin{equation*}
    K^\mathrm{bm}_t(x,y)=\frac{1}{2\pi t\sqrt{\det \Sigma}}
\left(1-\exp \left(-\frac{2y_2x_2}{t\Sigma_{22}}\right)\right)\exp\left(-\frac{\vert C(y-x-t\mu)\vert^2}{2t}\right),
\end{equation*}
where $C=\Sigma^{-1/2}$.
\end{lemma}
\begin{proof}[Proof of Lemma \ref{c01}]
We essentially follow the proof of \cite[Lemma 21]{HRT2023}. However, we correct some typos therein.
 For $v\in \mathbb{R}^2$ denote by $s_v$ the reflection with respect to $v$, that is, $s_v(x)=x-2\frac{\langle x,v^{\perp}\rangle}{\langle v^{\perp},v^{\perp}\rangle}v^{\perp}$ for any orthogonal to $v$  non-zero vector $v^{\perp}$. 
Denote $e_1=(1,0)$ and $e_2=(0,1)$.
For every Brownian motion 
$W$
and every $v\in \R^2$
we define
$$
\tau_x^{W,v}:=\inf\{t\geq 0: x+W(t)\in \R v\}.
$$
The reflection principle yields for all $t\in (0,\infty)$, $x,y\in \R\times [0,\infty)$
that if $W$ is a two-dimensional Brownian motion with zero drift and with unit covariance matrix, then
\begin{align}
\frac{d}{dy}&\P(x+W(t)\in dy, \tau_x^{W,v}>t)\nonumber\\
=&\frac{1}{2\pi t}\left(\exp\left(-\frac{\vert y-x\vert^2}{2t}\right)-\exp\left(-\frac{\vert s_v(y)-x\vert^2}{2t}\right)\right)\nonumber\\
=&\frac{1}{2\pi t}\left(\exp\left(-\frac{\vert y-x\vert^2}{2t}\right)-\exp\left(-\frac{\vert y-x-2\frac{\langle y,v^{\perp}\rangle}{\langle v^{\perp},v^{\perp}\rangle}v^{\perp}\vert^2}{2t}\right)\right)\nonumber\\
=&\frac{1}{2\pi t}\left(\exp\left(-\frac{\vert y-x\vert^2}{2t}\right)-\exp\left(-\frac{\vert y-x\vert^2-4\frac{\langle y-x,v^{\perp}\rangle\cdot\langle y,v^{\perp}\rangle}{\langle v^{\perp},v^{\perp}\rangle}+4\frac{\langle y,v^{\perp}\rangle^2}{\langle v^{\perp},v^{\perp}\rangle}}{2t}\right)\right)\nonumber\\
=&\frac{1}{2\pi t}
\left(1-\exp \left(-\frac{2\langle y,v^{\perp}\rangle\langle x,v^{\perp}\rangle}{t\langle v^{\perp},v^{\perp}\rangle} \right)\right)\exp\left(-\frac{\vert y-x\vert^2}{2t}\right).\label{c02}
\end{align}
Let us define 
$\widehat{B}(t)=CB(t)$ and $W(t)=\widehat{B}(t)-tC\mu$. The assumption that 
$B$ has drift $\mu$ and covariance matrix $\Sigma$ implies that
 $\widehat{B}(t)$ is a Brownian motion with drift $C\mu$ and with unit covariance matrix and, consequently, $W(t)$ is a standard Brownian motion. 
Using the change of variable $u\mapsto Cu$, the definitions of $\widehat{B}(t)$ and $W(t)$, and Girsanov theorem applied to $W(t)$, we obtain for all $t\in (0,\infty)$, $x,y\in \R\times[0,\infty)$ the equality
\begin{align*}
K_t^\mathrm{bm}(x,y)&=\frac{d}{dy}\P(x+B(t)\in dy, \tau^{bm}_x>t)
=\frac{d}{dy}\P(x+B(t)\in dy, \tau^{B,e_1}_x>t)\nonumber\\
&=(\det C)\frac{d}{dy}\P(Cx+\widehat{B}(t)\in Cdy, \tau_{Cx}^{\widehat{B},Ce_1}>t)\nonumber\\
&=(\det C)\exp\! \left(\langle C\mu, C(x-y)\rangle-\frac{t\vert C\mu\vert^2}{2}\right)\frac{d}{dy}\P(Cx+W(t)\in Cdy, \tau_{Cx}^{W,Ce_1}>t).
\end{align*}
Noting that $\langle \Sigma^{1/2}e_2,Ce_1\rangle=0$, and using \eqref{c02},
we obtain
\begin{align*}
&K_t^\mathrm{bm}(x,y)=(\det C)\exp\!\left(\langle C\mu,C (x-y)\rangle-\frac{t\vert C\mu\vert^2}{2}\right)\nonumber
\\&\quad \times\frac{1}{2\pi t}
\left(1-\exp\! \left(-\frac{2\langle Cy,C^{-1}e_2\rangle\langle Cx,C^{-1}e_2\rangle}{t \langle C^{-1}e_2,C^{-1}e_2\rangle}\right)\right)
\exp\!\left(-\frac{\vert Cy-Cx\vert^2}{2t}\right)\nonumber\\
=&\frac{1}{2\pi t\sqrt{\det \Sigma}}
\left(1-\exp (-\frac{2y_2x_2}{t\Sigma_{22}})\right)
\exp\!\left(-\frac{\vert Cy-Cx\vert^2-2t\langle C(y-x),C\mu\rangle+t^2\langle C\mu,C\mu\rangle}{2t}\right)\nonumber\\
=&\frac{1}{2\pi t\sqrt{\det \Sigma}}
\left(1-\exp (-\frac{2y_2x_2}{t\Sigma_{22}})\right)\exp\!\left(-\frac{\vert C(y-x-t\mu)\vert^2}{2t}\right).
\end{align*}
This completes the proof of Lemma~\ref{c01}.
\end{proof}
\begin{corollary}\label{c10}
Assume that $\mu=(\mu_1,0)\in\R^2 $ and
$\Sigma=\begin{pmatrix}
\sigma_1^2&\rho\\
\rho&\sigma_2^2
\end{pmatrix}
$.
Then
for $x=(0,x_2)$ and all $y=(y_1,y_2)\in \R\times(0,\infty)$,
as $x_2\downarrow 0$,
\begin{align*}
&\frac{d}{dy}\P\!\left(x+B(1)\in dy\middle| \tau^\mathrm{bm}_x>1\right)\nonumber\\&\sim \frac{y_2}{\sigma_2 \sqrt{2\pi(\sigma_1^2\sigma_2^2-\rho^2)}} \exp\!\left(-\frac{\sigma_2^2(y_1-\mu_1)^2+\sigma_1^2y_2^2-2\rho(y_1-\mu_1)y_2 }{2(\sigma_1^2\sigma_2^2-\rho^2)}\right).
\end{align*}

\end{corollary}

\begin{proof}
[Proof of Corollary~\ref{c10}]
Let $C=\Sigma^{-1/2}$, $e_1=(1,0)$, $e_2=(0,1)$.
First, the fact that for all $a,b,c,d\in \R$ with $ad\neq bc$ we have that
$
\begin{pmatrix}
a&b\\
c&d
\end{pmatrix}^{-1}
= \frac{1}{ad-bc}
\begin{pmatrix}
d&-b\\
-c&a
\end{pmatrix}
$ and the definition of $\Sigma$ show that
\begin{align*}
\Sigma^{-1}= \frac{1}{\sigma_1^2\sigma_2^2-\rho^2}
\begin{pmatrix}
\sigma_2^2&-\rho\\
-\rho&\sigma_1^2
\end{pmatrix}.
\end{align*}
This shows for all $h=(h_1,h_2)\in \R^2$ that
\begin{align*}
\left\langle
h,\Sigma^{-1}h\right\rangle
&
= 
\left\langle
\begin{pmatrix}
h_1\\h_2
\end{pmatrix}, 
\frac{1}{\sigma_1^2\sigma_2^2-\rho^2}
\begin{pmatrix}
\sigma_2^2&-\rho\\
-\rho&\sigma_1^2
\end{pmatrix}\begin{pmatrix}
h_1\\h_2
\end{pmatrix}\right\rangle
=
\left\langle
\begin{pmatrix}
h_1\\h_2
\end{pmatrix},\frac{1}{\sigma_1^2\sigma_2^2-\rho^2}
\begin{pmatrix}
\sigma_2^2h_1-\rho h_2\\
-\rho h_1+\sigma_1^2h_2
\end{pmatrix}
\right\rangle\nonumber\\
&=\frac{1}{\sigma_1^2\sigma_2^2-\rho^2}\left(\sigma_2^2 h_1^2-\rho h_1h_2-\rho h_1 h_2+\sigma_1^2h_2^2\right)
=\frac{\sigma_2^2h_1^2+\sigma_1^2 h_2^2-2\rho h_1h_2}{\sigma_1^2\sigma_2^2-\rho^2}.
\end{align*}
This and the fact that $\mu=(\mu_1,0)$ show that, for all $y=(y_1,y_2)\in \R\times(0,\infty)$,
\begin{align}
&
\Vert C(y-x-\mu)\Vert^2=\langle\Sigma^{-1} (y-x-\mu),(y-x-\mu)\rangle\nonumber\\
&=\frac{\sigma_2^2(y_1-x_1-\mu_1)^2+\sigma_1^2(y_2-x_2)^2-2\rho(y_1-x_1-\mu_1)(y_2-x_2) }{\sigma_1^2\sigma_2^2-\rho^2}\nonumber\\
&\to \frac{\sigma_2^2(y_1-\mu_1)^2+\sigma_1^2y_2^2-2\rho(y_1-\mu_1)y_2 }{\sigma_1^2\sigma_2^2-\rho^2},
\label{c03}
\end{align}
provided that $x_1=0$ and $x_2\downarrow 0$.

Next, 
using the relation
$1-e^{-a\varepsilon}\sim a\varepsilon$ as $\varepsilon\to 0$, we obtain
\begin{align*}
1-\exp \left(-\frac{2y_2x_2}{\Sigma_{22}}\right)=1-
\exp \left(-\frac{2y_2x_2}{\sigma_{2}^2}\right)
\sim \frac{2y_2x_2}{\sigma_{2}^2},\quad x_2\downarrow0.
\end{align*}
This, Lemma~\ref{c01} with $t=1$, \eqref{c03}, and the fact that
$\Sigma_{22}=1$ show that, for all $y=(y_1,y_2)\in \R\times(0,\infty)$, 
as $x_2\to 0$,
\begin{align}
&
\frac{d}{dy}\P\!\left(x+B(1)\in dy, \tau^\mathrm{bm}_x>1\right)=
\frac{1}{2\pi \sqrt{\det \Sigma}}
\left(1-\exp \!\left(-\frac{2y_2x_2}{\Sigma_{22}}\right)\right)\exp\left(-\frac{\Vert C(y-x-\mu)\Vert^2}{2}\right)\nonumber\\
&\sim
\frac{1}{2\pi \sqrt{\sigma_1^2\sigma_2^2-\rho^2}} \frac{2y_2x_2}{\sigma_{2}^2}\exp\!\left(-\frac{\sigma_2^2(y_1-\mu_1)^2+\sigma_1^2y_2^2-2\rho(y_1-\mu_1)y_2 }{2(\sigma_1^2\sigma_2^2-\rho^2)}\right).\label{c04}
\end{align}
Note that $B_2/\sigma_2$ is a standard $1$-dimensional Brownian motion. This and the reflection principle 
show that for all $x=(x_1,x_2)\in \R\times(0,\infty)$, as $x_2\downarrow 0$,
\begin{align*}
\P(\tau_x^B> 1)
&= \P\!\left(\min_{t\in [0,1]} B_2(t)>-x_2 \right)
= \P\!\left(\max_{t\in [0,1]} B_2(t)<x_2 \right)\nonumber\\
&=\P\!\left(\max_{t\in [0,1]} \frac{B_2(t)}{\sigma_2}<\frac{x_2}{\sigma_2} \right)=
1-\P\!\left(\max_{t\in [0,1]} \frac{B_2(t)}{\sigma_2}\geq \frac{x_2}{\sigma_2} \right)
\nonumber\\
&=1-2\P\!\left(B_2(1)\geq \frac{x_2}{\sigma_2}\right )=2\P\!\left(B_2(1)\in \left(0,\frac{x_2}{\sigma_2}\right)\right)\nonumber\\&=
2
\frac{1}{\sqrt{2\pi}}
\int_0^{x_2/\sigma_2} e^{-\frac{u^2}{2}}\,du\sim \sqrt{\frac{2}{\pi}}\frac{x_2}{\sigma_2}.
\end{align*}
Combining this with \eqref{c04}, we get the desired result.
\end{proof}
\section{Random walks in the upper half-plane.}\label{s03}
The main purpose of this section is to prove a local limit theorem for lattice random walks with drift conditioned to stay in the upper half-plane. The case of walks without drift has been studied in \cite{DW2024} and in \cite{DW2015}, where general cones have been considered. The case of non-zero drift does not require significant changes. We prefer to give a rather complete proof, since we do not assume that the walk is strongly aperiodic. This assumption was made in \cite{DW2015,DW2024}. As it has been mentioned in \cite{DW2019}, the periodicity is not an issue. But it makes sense to have a written proof.

Let us start with an integral limit theorem for conditioned random walks.
\begin{proposition}\label{t05}
Assume that the conditions of Theorem~\ref{thm:tail_and_limit} are valid.
Let $p$  be the density function obtained in Corollary~\ref{c10} with $\mu_1=0$:
\begin{align}
\label{eq:p-def}
   p(y)= \frac{y_2}{\sigma_2 \sqrt{2\pi(\sigma_1^2\sigma_2^2-\rho^2)}} \exp\!\left(-\frac{\sigma_2^2y_1^2+\sigma_1^2y_2^2-2\rho y_1y_2 }{2(\sigma_1^2\sigma_2^2-\rho^2)}\right)
\end{align} for $ y\in\R\times(0,\infty).$
Then, for every fixed $x\in\R\times(0,\infty)$ one has 
\begin{align*}
\lim_{n\to\infty}
\P \!\left(\frac{x+S(n)-n\mu}{\sqrt{n}}\in u+\Delta\middle| \tau_x>n\right)= \int_{u+\Delta}p(y)\,dy,
\end{align*} 
where $\Delta= [0,1)^d$.
\end{proposition}
This result is immediate consequence of Theorem 1 in \cite{DW2024} applied to the walk $S(n)-\mu n$.

Next, we collect some estimates for the tail of $\tau_x$ and for local probabilities $\P(x+S(n)=z)$ and $\P(x+S(n)=z,\tau_x>n)$, which will play a crucial role in the proofs of our main results. 
\begin{lemma}
\label{lem:loc.bound.1} 
Assume that \eqref{eq:1st.moment} and \eqref{eq:2nd.moment} are valid.
\begin{enumerate}[(i)]
\item There exists a constant $C$ such that
    \begin{align}
    \label{eq:tau-x}
    \P(\tau_x>n)\le C\frac{V(x_2)}{\sqrt{n}}
    \end{align}
    for all $x\in\R\times(0,\infty)$ and all $n\ge1$. Furthermore, for every fixed $x\in\R\times(0,\infty)$,
    \begin{align}
    \label{eq:tau_x-asymp}
    \lim_{n\to\infty}\sqrt{n}\P(\tau_x>n)=\varkappa V(x_2).
    \end{align}
\item If $S(n)$ takes values on the lattice $\Z^2$ then there exists a constant $C$ such that 
    \begin{align}
    \label{eq:loc.b.1}
    \P(x+S(n)=z,\tau_x>n)\le C\frac{V(x_2)}{n^{3/2}}
    \end{align}
    for all $x,z\in\Z\times\N$ and all $n\ge1$.
 \item If $S(n)$ takes values on the lattice $\Z^2$ then there exists a constant $C$ such that 
    \begin{align}
    \label{eq:loc.b.2}
    \P(x+S(n)=z,\tau_x>n)\le C\frac{V(x_2)H(z_2)}{n^{2}}
    \end{align}   
    for all $x,z\in\Z\times\N$ and all $n\ge1$.
\end{enumerate}
\end{lemma}
\begin{proof}
The first claim is immediate from Lemma~6 in \cite{DW2024} and from the fact that $\P(\tau_0>n)\sim C_1n^{-1/2}$ as $n\to\infty$. 

To get the estimates \eqref{eq:loc.b.1} and \eqref{eq:loc.b.2} we apply Lemma~8 and Lemma~7 from \cite{DW2024} to the driftless random walk $S(n)-\mu n$ and notice that $c_n=\sqrt{n}$ under the assumption that the walk $S(n)$ has finite variances.
\end{proof}

Combining \eqref{eq:loc.b.1} and \eqref{eq:loc.b.2}
and recalling that the function $H(x)$ is asymptotically linear, we get the following bound.
\begin{corollary}
\label{d06}
Then there exist constants $C_1,C_2\in (0,\infty)$ such that for all $x=(x_1,x_2), y=(y_1,y_2)\in \Z\times \N$, $n\in \N$ we have that
\begin{align*}
\P(x+S(n)=y,\tau_x>n)\leq \frac{C_1V(x_2)H(\min\{\sqrt{n},y_2\})}{n^2}
\leq \frac{C_2V(x_2)\min\{\sqrt{n},y_2\}}{n^2}.
\end{align*}
\end{corollary}
Furthermore, we shall use the following bounds for unconditional local probabilities
\begin{lemma}[cf.\ {\cite[Lemma~29]{DW2015}}]\label{d04}
Assume that \eqref{eq:1st.moment} and \eqref{eq:2nd.moment} are valid. Then there exist $C_1,C_2\in (0,\infty)$ such that for every $u\in (0,\infty)$ we have that
\begin{align*}
\limsup_{n\to\infty} \sup_{\lvert z-x-n\mu \rvert \geq u\sqrt{n}}\P(x+S(n)=z)\leq C_2n^{-1}e^{-C_1u^2}.
\end{align*}
\end{lemma}
\begin{proof}[Proof of Lemma~\ref{d04}]
Apply \cite[(75) in Lemma~29]{DW2015} to the driftless random walk $S(n)-\mu n$.
\end{proof}
\begin{lemma}[cf.\ {\cite[Lemma~29]{DW2015}}]\label{d04b}Assume that \eqref{eq:1st.moment} and \eqref{eq:2nd.moment} are valid. Then there exist $C_1,C_2\in (0,\infty)$ such that for every $u\in (0,\infty)$ we have that
\begin{align*}
\limsup_{n\to\infty} \sup_{x_2,z_2>u\sqrt{n}}\P(x+S(n)=z,\tau_x\le n)\leq C_2n^{-1}e^{-C_1u^2}.
\end{align*}
\end{lemma}
\begin{proof}[Proof of Lemma~\ref{d04b}]
Apply \cite[(76) in Lemma~29]{DW2015} to the walk $S(n)-\mu n$ and notice that subtracting the drift we do not change the vertical component of the walk.
\end{proof}
To formulate a local limit theorem for random walks conditions to stay in the upper half-plane we introduce
$$
\overline{D}_n(x)=\{y\in\Z\times\N:\, y-x\in D_n\}.
$$
\begin{proposition}
\label{d07}Assume that the conditions of Theorem~\ref{t11} are valid. The, for every $x\in\Z\times\N$ we have that
\begin{align*}
\lim_{n\to\infty}
\sup_{y\in \overline{D}_n(x)} 
\left\lvert n^{3/2}\P(x+S(n)=y,\tau_x>n)-\varkappa d_1d_2p\!\left(\frac{y-n\mu }{\sqrt{n}}\right)V(x_2)\right\rvert=0.
\end{align*}
Furthermore.
$$
\P(x+S(n)=y,\tau_x>n)=0\quad\text{for all }y\notin \overline{D}_n(x).
$$
\end{proposition}

\begin{proof}[Proof of Proposition~\ref{d07}]
Set 
\begin{align*} \begin{split} 
&
H^{(1)}=
H^{(1)}_{A,n}=\left\{y\in \Z\times\N\colon\lvert y-n\mu\rvert> A\sqrt{n}\right\},\\
&
H^{(2)}=
H^{(2)}_{n}=\left\{y\in \Z\times\N\colon y_2\leq  2\varepsilon\sqrt{n}\right\},\\
&
H^{(3)}=
H^{(3)}_{A,n}=\left\{y\in \Z\times\N\colon y_2>2\varepsilon\sqrt{n},\lvert y-n\mu\rvert\leq  A\sqrt{n}\right\},\\
\end{split}
\end{align*}
From Corollary~\ref{d06}  it follows that there exists $C\in (0,\infty)$ such that for all $n\in \N$, $y\in H^{(2)}_n$,
\begin{align*}
\P(x+S(n)=y,\tau_x>n)\leq \frac{CV(x_2)\varepsilon}{n^{3/2}}.
\end{align*}
Consider now the case when $y\in H^{(1)}$.
Using the Markov property and applying the concentration bound for the walk $S(n)$, we have
\begin{align*}
&\P\!\left(x+S(n)=y,\tau_x>n, \lvert x+S(n/2)-\mu n/2\rvert> \frac{A\sqrt{n}}{2}\right)\nonumber\\
&=\sum_{z\in \Z\times\N}
\P\!\left(x+S(n/2)=z,\tau_x>n/2, \lvert x+S(n/2)-\mu n/2\rvert> \frac{A\sqrt{n}}{2}\right)\nonumber\\&\qquad\qquad\qquad\P\!\left(z+S(n/2)=y,\tau_z> n/2\right)\nonumber\\
&\leq \sum_{z\in \Z\times\N}\P\!\left(x+S(n/2)=z,\tau_x>n/2, \lvert x+S(n/2)-\mu n/2\rvert> \frac{A\sqrt{n}}{2}\right)\frac{C}{n}\nonumber\\
&=\P\!\left(\tau_x>n/2, \lvert x+S(n/2)-\mu n/2\rvert>
 \frac{A\sqrt{n}}{2}\right)\frac{C}{n}\nonumber\\
&=
\P\!\left( \lvert x+S(n/2)-\mu n/2\rvert>
 \frac{A\sqrt{n}}{2}\middle| \tau_x>n/2\right)\P(\tau_x>n/2)\frac{C}{n}\nonumber\\
&\leq \P\!\left( \lvert x+S(n/2)-\mu n/2\rvert>
 \frac{A\sqrt{n/2}}{\sqrt{2}}\middle| \tau_x>n/2\right)\frac{V(x_2)}{n^{1/2}}\frac{C}{n}.
\end{align*}
Thus, the convergence in Proposition~\ref{t05} shows that
\begin{align}
&
\lim_{A\to\infty}
\limsup_{n\to\infty}\sup_{y\in H^{(1)}_{A,n}}\frac{n^{3/2}}{V(x_2)}
\P\!\left(x+S(n)=y,\tau_x>n, \lvert x+S(n/2)-\mu n/2\rvert> \frac{A\sqrt{n}}{2}\right)\nonumber\\
&\leq 
\lim_{A\to\infty}
\limsup_{n\to\infty}\frac{n^{3/2}}{V(x_2)}
\P\!\left( \lvert x+S(n/2)-\mu n/2\rvert>
 \frac{A\sqrt{n/2}}{\sqrt{2}}\middle| \tau_x>n/2\right)\frac{V(x_2)}{n^{1/2}}\frac{C}{n}\nonumber\\
&\leq \lim_{A\to\infty}C
\int_{\{z\in(0,\infty)\times \R\colon \lvert z\rvert\geq A/\sqrt{2}\}}p(y)dy=0.\label{d09}
\end{align}
Furthermore, Lemmas~\ref{lem:loc.bound.1}  and~\ref{d04} show that there exist $C,C_1\in (0,\infty)$ such that for all $x=(x_1,x_2),y\in H^{(1)}$ we have that
\begin{align*}
&\P\!\left(x+S(n)=y,\tau_x>n, \lvert x+S(n/2)-\mu n/2 \rvert\leq \frac{A\sqrt{n}}{2}\right)\nonumber\\
&\leq \P(\tau_x>n/2)\sup_{z\in \N\times\Z\colon \lvert y-z-\mu n/2\rvert\geq \frac{A\sqrt{n}}{2}}
\P(z+S(n/2)=y)\leq C\frac{V(x_2)}{n^{3/2}} e^{-C_1A^2/4}.
\end{align*}
Combining this with \eqref{d09} proves for all $x\in \Z\times\N$ that
\begin{align*}
&\lim_{A\to\infty}
\limsup_{n\to\infty}\sup_{y\in H^{(1)}_{A,n}}
\frac{n^{3/2}
\P(x+S(n)=y,\tau_x>n)}{V(x_2)}\nonumber\\
&
=\lim_{A\to\infty}
\limsup_{n\to\infty}\sup_{y\in H^{(1)}_{A,n}}
\frac{n^{3/2}}{V(x_2)}
\P\!\left(x+S(n)=y,\tau_x>n, \lvert x+S(n/2)-\mu n/2\rvert\leq \frac{A\sqrt{n}}{2}\right)\nonumber\\
&\quad +\lim_{A\to\infty}
\limsup_{n\to\infty}\sup_{y\in H^{(1)}_{A,n}}
\frac{n^{3/2}}{V(x_2)}
\P\!\left(x+S(n)=y,\tau_x>n, \lvert x+S(n/2)-\mu n/2\rvert> \frac{A\sqrt{n}}{2}\right)\nonumber
\\
&=\lim_{A\to \infty} Ce^{-C_1A^2/4}+0=0.
\end{align*}
Set $m=\lfloor \varepsilon^3n\rfloor$. By the Markov property,
\begin{align*}
&\P (x+S(n)=y,\tau_n>n)\nonumber\\
&
= \sum_{z\in \Z\times\N} \P (x+S(n-m)=z,\tau_x>n-m)
\P(z+ S(m)=y,\tau_z>m).
\end{align*}
Let $H_\varepsilon(y)= \{z\in \N\times\Z\colon \lvert z-y\rvert<\varepsilon\sqrt{n}\}$. Using Lemma~\ref{d04} we have
\begin{align*}
& \sum_{z\in (\Z\times\N)\setminus H_\varepsilon(y)} \P (x+S(n-m)=z,\tau_x>n-m)
\P(z+ S(m)=y,\tau_z>m)\nonumber\\
&\leq \P(\tau_x> n-m)\sup_{z\in (\Z\times\N)\setminus H_\varepsilon(y)}
\P(z+S(m)=y)\nonumber\\
&\leq C\frac{V(x_2)}{\sqrt{n}}m^{-1}e^{-c_2/\varepsilon}=C \frac{V(x_2)}{n^{3/2}}\varepsilon^{-3}e^{-C_2 /\varepsilon}.
\end{align*}
If $y\in H^{(3)}$ and $z\in H_\varepsilon(y)$, then 
$z_2>\varepsilon \sqrt{n}$. 
Then 
Lemmas~\ref{d04b} and \ref{lem:loc.bound.1}  show that
\begin{align*}
&\sum_{z\in H_\varepsilon(y)}\P (x+S(n-m) =z,\tau_x>n-m)\P(z+S(m)=y,\tau_z\leq m)\nonumber\\
&\leq 
\sum_{z\in H_\varepsilon(y)}
\P (x+S(n-m) =z,\tau_x>n-m)C_1m^{-1}e^{-C_2\varepsilon}
\nonumber\\
&\leq \P (\tau_x>n-m)C_1m^{-1}e^{-C_2\varepsilon}\leq C_1\frac{V(x_2)}{n^{3/2}} \varepsilon^{-1}e^{-C_2/\varepsilon}.
\end{align*}
Similarly,
\begin{align*}
&\sum_{z\in H_\varepsilon(y)}\P (x+S(n-m) =z,\tau_x\leq n-m)\P(z+S(m)=y,\tau_z> m)\nonumber\\
&\leq 
\sum_{z\in H_\varepsilon(y)}C_1n^{-1}e^{-C_2}
\P(z+S(m)=y,\tau_z> m)\nonumber\\
&\leq C_1n^{-1}e^{-C_2}\frac{V(x_2)}{\sqrt{m}}.
\end{align*}
The remaining part goes as in the proof of \cite[Theorem~5]{DW2015}.
\end{proof}
\section{Tail of $T_x$ and the conditional limit theorem}\label{s04}
In this section we prove Theorem~\ref{thm:tail_and_limit}.
\subsection{Tail asymptotics for $T_x$.}
Notice that $T_x=\min\{\tau_x,\sigma_x\}$. Then, by the total probability law, for every fixed $N$ and all $n>N$,
\begin{align}
\label{eq:tail.1}
\nonumber
\P(T_x>n)
&=\P(\tau_x>n,\sigma_x>n)\\
\nonumber
&=\P(\tau_x>n)-\P(\tau_x>n,\sigma_x\le n)\\
&=\P(\tau_x>n)-\P(\tau_x>n,\sigma_x\le N)-\P(\tau_x>n,\sigma_x\in(N, n]).
\end{align}
Combining \eqref{eq:tau-x} and \eqref{eq:tau_x-asymp}, and applying the the dominated convergence theorem, we conclude that 
\begin{align}
\label{eq:tail.2}
\nonumber
&\lim_{n\to\infty}\sqrt{n}\P(\tau_x>n,\sigma_x=k)\\
\nonumber
&\hspace{1cm}=\lim_{n\to\infty}\sqrt{n}
\int_{\R\times(0,\infty)}\P(x+S(k)\in dz,\tau_x>k,\sigma_x=k)\P(\tau_z>n-k)\\
\nonumber
&\hspace{1cm}=\varkappa\int_{\R\times(0,\infty)}
\P(x+S(k)\in dz,\tau_x>k,\sigma_x=k)V(z_2)\\
&\hspace{1cm}=\varkappa\E[V(x_2+S_2(\sigma_x));\tau_x>k,\sigma_x=k].
\end{align}
Consequently, for every fixed $N$,
\begin{equation}
\label{eq:tail.3}
\lim_{n\to\infty}\sqrt{n}\P(\tau_x>n,\sigma_x\le N)
=\varkappa\E[V(x_2+S_2(\sigma_x));\tau_x>\sigma_x,\sigma_x\le N].
\end{equation}
Combining this with \eqref{eq:tail.1} and using \eqref{eq:tau_x-asymp}, we conclude that 
$$
\limsup_{n\to\infty}\sqrt{n}\P(T_x>n)
\le \varkappa V(x_2)
-\varkappa\E[V(x_2+S_2(\sigma_x));\tau_x>\sigma_x,\sigma_x\le N].
$$
Letting here $N\to\infty$ and using the monotone convergence theorem, we obtain
\begin{equation}
\label{eq:tail.4} 
\limsup_{n\to\infty}\sqrt{n}\P(T_x>n)
\le \varkappa V(x_2)
-\varkappa\E[V(x_2+S_2(\sigma_x));\tau_x>\sigma_x,\sigma_x<\infty].
\end{equation}
This relation implies also that 
\begin{align}
\label{eq:tail.5}
\E[V(x_2+S_2(\sigma_x));\tau_x>\sigma_x,\sigma_x<\infty]\le V(x_2),
\quad x\in\R\times\R_+.
\end{align}
Next, using \eqref{eq:tau-x} once again, we obtain 
\begin{align*}
\P(\tau_x>n,\sigma_x\in(N,n/2])
&=\sum_{k=N+1}^{n/2}\int_{\R\times(0,\infty)}
\P(x+S(k)\in dz,\tau_x>k,\sigma_x=k)\P(\tau_z>n-k)\\
&\le\frac{C}{\sqrt{n}}\sum_{k=N+1}^{n/2}\int_{\R\times(0,\infty)}
\P(x+S(k)\in dz,\tau_x>k,\sigma_x=k)V(z_2)\\
&=\frac{C}{\sqrt{n}}\E[V(x_2+S_2(\sigma_x));\tau_x>\sigma_x,\sigma_x\in(N,n/2]]\\
&\le\frac{C}{\sqrt{n}}
\E[V(x_2+S_2(\sigma_x));\tau_x>\sigma_x,\sigma_x\in(N,\infty)].
\end{align*}
Taking into account the integrability of $V(x_2+S_2(\sigma_x)){\rm 1}\{\tau_x>\sigma_x,\sigma_x<\infty\}$, see \eqref{eq:tail.5}, we obtain 
\begin{align}
\label{eq:tail.6}
\lim_{N\to\infty}\limsup_{n\to\infty}
\sqrt{n}\P(\tau_x>n,\sigma_x\in(N,n/2])=0.
\end{align}
Finally,
\begin{align*}
&\P(\tau_x>n,\sigma_x\in(n/2,n])\\
&\hspace{1cm}\le \P(\sigma_x\in(n/2,n])\\
&\hspace{1cm}\le \P\!\left(x_1+S_1(n/2)\leq \mu_1n/4\right)
+\P\!\left(\min_{k\in (n/2,n]} (S_1(k)-S_1(n/2))\leq -\mu_1 n/4 \right)\\
&\hspace{1cm}\le \P(x_1+S_1(n/2)\le\mu n/4)+\P\!\left(\min_{k\le n/2}S_1(k)\le-\mu n/4\right).
\end{align*}
Applying now the Doob inequality and noting that the walk $S_1(k)-\mu_1k$ is driftless, we conclude that
\begin{equation}
 \label{eq:tail.7}
 \P(\tau_x>n,\sigma_x\in(n/2,n])\le\frac{C}{n}.
\end{equation}
Combining now \eqref{eq:tail.2}, \eqref{eq:tail.6} and \eqref{eq:tail.7}, we conclude that 
\begin{align*}
\lim_{n\to\infty}\sqrt{n}\P(T_x>n)
&=\varkappa\left(V(x_2)
-\varkappa\E[V(x_2+S_2(\sigma_x));\tau_x>\sigma_x,\sigma_x<\infty]\right)\\
&=\varkappa W(x).
\end{align*}
Thus, \eqref{eq:T_tail} is proved.
\subsection{Conditional limit theorem.}
Similar to the proof of \eqref{eq:T_tail},
\begin{align*}
&\P \!\left(\frac{x+S(n)-n\mu}{\sqrt{n}}\in u+\Delta;T_x>n\right)\\
&\hspace{1cm}
=\P \!\left(\frac{x+S(n)-n\mu}{\sqrt{n}}\in u+\Delta;\tau_x>n\right)
-\P \!\left(\frac{x+S(n)-n\mu}{\sqrt{n}}\in u+\Delta;\tau_x>n,\sigma_x\le n\right)\\
&\hspace{1cm}
=\P \!\left(\frac{x+S(n)-n\mu}{\sqrt{n}}\in u+\Delta;\tau_x>n\right)
-\P \!\left(\frac{x+S(n)-n\mu}{\sqrt{n}}\in u+\Delta;\tau_x>n,\sigma_x\le N\right)\\
&\hspace{3cm} -\P \!\left(\frac{x+S(n)-n\mu}{\sqrt{n}}\in u+\Delta;\tau_x>n,\sigma_x\in(N,n]\right).
\end{align*}
It is immediate from \eqref{eq:tail.6} and \eqref{eq:tail.7} that
\begin{align}
\label{eq:limit.1}
\nonumber
&\lim_{N\to\infty}\limsup_{n\to\infty}
\sqrt{n}\P \!\left(\frac{x+S(n)-n\mu}{\sqrt{n}}\in u+\Delta;\tau_x>n,\sigma_x\in(N,n]\right)\\
&\hspace{1cm}
\le \lim_{N\to\infty}\limsup_{n\to\infty}\sqrt{n}\P(\tau_x>n,\sigma_x\in(N,n])
=0.
\end{align}
Furthermore, using Proposition~\ref{t05} and the fact that
$\P(\tau_x>n)\sim \varkappa \frac{V(x_2)}{\sqrt{n}}$, we have that
$$
\sqrt{n}\P \!\left(\frac{x+S(n)-n\mu}{\sqrt{n}}\in u+\Delta;\tau_x>n\right)
\sim \varkappa V(x_2)\int_{u+\Delta} p(y)dy
$$
and
\begin{align*}
&\sqrt{n}
\P \!\left(\frac{x+S(n)-n\mu}{\sqrt{n}}\in u+\Delta;\tau_x>n,\sigma_x=k\right)\\ 
&=\sqrt{n}\int \P(x+S(k)\in dz,\tau_x>k,\sigma_x=k)
\P \!\left(\frac{z+S(n-k)-n\mu}{\sqrt{n}}\in u+\Delta;\tau_z>n-k\right)\\
&\hspace{1cm}
\sim \varkappa\E[V(x_2+S_2(k));\tau_x>k,\sigma_x=k]\int_{u+\Delta}p(y)dy
\end{align*}
for every fixed $k$. Therefore,
\begin{align*}
&\sqrt{n}\left(\P \!\left(\frac{x+S(n)-n\mu}{\sqrt{n}}\in u+\Delta;\tau_x>n\right)
-\P \!\left(\frac{x+S(n)-n\mu}{\sqrt{n}}\in u+\Delta;\tau_x>n,\sigma_x\le N\right)\right) \\
&\hspace{2cm}\sim\varkappa
\left(V(x_2)-\E[V(x_2+S_2(k));\tau_x>\sigma_x,\sigma_x\le N]\right)
\int_{u+\Delta}p(y)dy.
\end{align*}
Letting here $N\to\infty$ and taking into account \eqref{eq:limit.1}, we conclude that 
\begin{align*}
&\sqrt{n}\P \!\left(\frac{x+S(n)-n\mu}{\sqrt{n}}\in u+\Delta;T_x>n\right)\\
&\hspace{1cm}\sim \varkappa\bigl(V(x_2)-\E[V(x_2+S_2(k));\tau_x>\sigma_x,\sigma_x<\infty]\bigr)\int_{u+\Delta}p(y)dy.
\end{align*}
This completes the proof of the conditonal limit theorem.
\subsection{Properties of the function $W$.}
The harmonicity of $V$ implies that the sequence
$V(x_2+S_2(n)){\rm 1}\{\tau_x>n\}$ is a martingale. Then, by the optional stopping theorem,
\begin{align*}
V(x_2)
&=\E[V(x_2+S_2(\sigma_x\wedge n)){\rm 1}\{\tau_x>\sigma_x\wedge n\}]\\
&=\E[V(x_2+S_2(\sigma_x)){\rm 1}\{\tau_x>\sigma_x\};\sigma_x\le n]
+\E[V(x_2+S_2(n)){\rm 1}\{\tau_x>n\};\sigma_x>n]\\
&=\E[V(x_2+S_2(\sigma_x)){\rm 1}\{\tau_x>\sigma_x\};\sigma_x\le n]
+\E[V(x_2+S_2(n));T_x>n].
\end{align*}
Consequently, by the monotone convergence theorem,
\begin{align*}
&\lim_{n\to\infty}\E[V(x_2+S_2(n));T_x>n]\\
&\hspace{1cm}=V(x_2)-\lim_{n\to\infty}
\E[V(x_2+S_2(\sigma_x)){\rm 1}\{\tau_x>\sigma_x\};\sigma_x\le n]\\
&\hspace{1cm}=V(x_2)
-\E[V(x_2+S_2(\sigma_x)){\rm 1}\{\tau_x>\sigma_x\};\sigma_x<\infty]\\
&\hspace{1cm}=W(x).
\end{align*}
This representation implies that the function $W(x)$ is harmonic for $S(n)$ killed at leaving $\R_+^2$, that is,
\begin{align*}
 W(x)=\E[W(x+S(1));T_x>1], \quad   x\in\R_+^2. 
\end{align*}

Since $V$ is harmonic to $S(n)$ killed at leaving the half-plane
$\R\times(0,\infty)$, we may define a new measure $\widehat\P$ with the transition kernel
$$
\widehat{P}(x,dy)=\frac{V(y_2)}{V(x_2)}\P(x+S(1)\in dy),
\quad x,y\in\R\times(0,\infty).
$$
This implies that, for every fixed $k$,
\begin{align*}
 \widehat{\P}(\sigma_x=k)
 =\frac{1}{V(x_2)}\E[V(x+S_2(k));\tau_x>\sigma_x=k]
\end{align*}
and, consequently,
$$
\widehat{\P}(\sigma_x<\infty)
=\frac{1}{V(x_2)}\E[V(x+S_2(\sigma_x));\tau_x>\sigma_x,\sigma_x<\infty].
$$
This yields
\begin{align}
\label{eq:W-V-repr}
 W(x)=V(x_2)\widehat{\P}(\sigma_x=\infty) ,\quad x,y\in\R_+^2. 
\end{align}
Thus, to show that $W(x)$ is positive, it suffices to prove that
$\widehat{\P}(\sigma_x=\infty)>0$.

We first show that $\widehat{\P}(\sigma_x=\infty)>0$ for all points $x$ with sufficiently large coordinates. Since the chain $S_2(n)$ is transient under the measure $\widehat{\P}$, for all $M$ and $\varepsilon>0$ there exists
$N_1=N_1(M,\varepsilon)$ such that 
\begin{equation}
\label{eq:sigma-1}
\widehat{\P}\left(\min_{k\ge1}S_2(k)\le M|S_2(0)=r\right)\le \varepsilon
\quad\text{for all }r\ge N_1.
\end{equation}

Assume now that we have constructed a positive, monotone decreasing to zero function 
$g(u)$ such that 
\begin{align}
 \label{eq:sigma-2}
 \widehat{\E}[g(S_1(1))-g(x_1)|S(0)=x]\le 0
\end{align}
for all $x$ with $x_2>M$. Set 
$$
\theta_M:=\inf\{k\ge0: S_2(k)\le M\}.
$$
Then the sequence $Y(n):=g(S_1(n)){\rm 1}\{\theta_M>n\}$ is a positive supermartingale. Indeed, for all $x$ one has 
$$
\widehat{E}[Y(1)-Y(0)|S(0)=x]=0\quad\text{if }x_2\le M
$$
and 
\begin{align*}
\widehat{E}[Y(1)-Y(0)|S(0)=x]
&=\widehat{E}[g(S_1(1)){\rm 1}\{\theta_M>1\}-g(S_1(0))|S(0)=x]\\
&\le \widehat{E}[g(S_1(1))-g(S_1(0))|S(0)=x]\le0
\quad\text{if }x_2>M.
\end{align*}
In the last step we have used \eqref{eq:sigma-2}. Applying now the Doob inequality for supermartingales, we have 
\begin{align*}
\widehat{\P}(\sigma_x<\infty)
&=\widehat{\P}\left(\min_{n\ge1}S_1(n)\le M\Big|S(0)=x\right)\\
&\le \widehat{\P}(\theta_M<\infty|S(0)=x)+
\widehat{\P}\left(\min_{n\ge1}S_1(n)\le M,\theta_M=\infty\Big|S(0)=x\right)\\
&\le \widehat{\P}(\theta_M<\infty|S(0)=x)+
\widehat{\P}\left(\max_{n\ge 1}Y(n)\ge g(M)|S(0)=x\right)\\
&\le \widehat{\P}(\theta_M<\infty|S(0)=x)+
\frac{g(x_1)}{g(M)}.
\end{align*}
Taking into account \eqref{eq:sigma-1}, we conclude that there exists $N_2=N_2(M,\varepsilon)$ such that 
\begin{align}
\label{eq:sigma-main}
\widehat{\P}(\sigma_x<\infty)\le 2\varepsilon
\end{align}
for all $x$ such that $x_1,x_2>\max\{N_1,N_2\}$. 

We next construct a function $g(u)$ satisfying \eqref{eq:sigma-2}. 
To this end we shall use the techniques from the proof of Theorem 2.21 in \cite{Denisov_Korshunov_Wachtel_2025}.
By the definition  of the measure $\widehat{\P}$,
\begin{align*}
\widehat{\P}(S_1(1)-S_1(0)\le -u|S(0)=x)
&=\frac{1}{V(x_2)}\E[V(x_2+X_2);X_1\le-u]\\
&=\P(X_1\le-u)+\frac{1}{V(x_2)}\E[(V(x_2+X_2)-V(x_2);X_1\le-u].
\end{align*}
Since the function $V$ is increasing, subadditive and is bounded above by a linear function, we obtain
\begin{align*}
\E[(V(x_2+X_2)-V(x_2);X_1\le-u]
&\le \E[V(X_2);X_1\le-u, X_2>0]\\
&\le C_1\left(\P(X_1\le-u)+\E[X_2;X_1\le-u, X_2>0] \right).
\end{align*}
Consequently,
\begin{align*}
\widehat{\P}(S_1(1)-S_1(0)\le -u|S(0)=x)
\le  C \left(\P(X_1\le-u)+\E[X_2;X_1\le-u, X_2>0] \right).
\end{align*}
The finiteness of second moments implies that there exists a decreasing integrable function $p(u)$ and an increasing function $s(u)=o(u)$ such that 
\begin{align}
\label{eq:sigma-3}
\widehat{\P}(S_1(1)-S_1(0)\le -s(u)|S(0)=x)
=o(p(u)),\quad u\to\infty.
\end{align}
According to \cite{Denisov2006}, there exists a continuous decreasing integrable function $p_1(u)$, regularly varying at infinity with index $-1$ and such that $p(u)\le p_1(u)$. Set now
$$
f(v):=\int_v^\infty p_1(u) du,\quad v\ge0.
$$
This function is slowly varying at infinity.

It is clear that 
\begin{align*}
&\widehat{\E}_x[f(S_1(1))-f(x_1)]\\
&\hspace{1cm}\le \widehat{\E}_x[f(S_1(1))-f(x_1);S_1(1)\le x_1+s(x_1)]\\
&\hspace{1cm}\le f(0)\widehat{\P}_x(S_1(1)-S_1(0)\le -s(x_1))
+\widehat{\E}_x[f(S_1(1))-f(x_1);|S_1(1)-x_1|\le s(x_1)].
\end{align*}
Taking in to account \eqref{eq:sigma-3} and applying the mean value theorem, we obtain 
\begin{align} 
\label{eq:sigma-4}
\widehat{\E}_x[f(S_1(1))-f(x_1)]
\le o(p(x_1))-p_1(x_1)(1+o(1))
\widehat{\E}_x[S_1(1)-x_1;|S_1(1)-x_1|\le s(x_1)].
\end{align}
By the definition of $\widehat{\P}$,
\begin{align*}
&\widehat{\E}_x[S_1(1)-x_1;|S_1(1)-x_1|\le s(x_1)]\\
&\hspace{1cm} =\frac{1}{V(x_2)}\E[V(x_2+X_2)X_1;|X_1|\le s(x_1)]\\
&\hspace{1cm} =\E[X_1;|X_1|\le s(x_1)]
+\frac{1}{V(x_2)}\E[(V(x_2+X_2)-V(x_2))X_1;|X_1|\le s(x_1)]\\
&\hspace{1cm}=\mu_1-\E[X_1;|X_1|>s(x_1)]
+\frac{1}{V(x_2)}\E[(V(x_2+X_2)-V(x_2))X_1;|X_1|\le s(x_1)].
\end{align*}
Recalling that $V$ is subadditive and bounded by a linear function, and using the Markov inequality, we conclude that
\begin{align*}
\widehat{\E}_x[S_1(1)-x_1;|S_1(1)-x_1|\le s(x_1)]
\ge \mu_1-\frac{\E[|X_1|^2]}{s(x_1)}-\frac{C\E[|X_1X_2|]}{V(x_2)}
\end{align*}
for some finite constant $C$. Thus, there exist $M$ such that 
\begin{align}
\label{eq:sigma-5}
\widehat{\E}_x[S_1(1)-x_1;|S_1(1)-x_1|\le s(x_1)]
\ge\frac{\mu_1}{2}\quad\text{for all }x\text{ such that }
x_1,x_2>M.
\end{align}
Combining \eqref{eq:sigma-4}, \eqref{eq:sigma-5} and recalling that $p(u)\le p_1(u)$, we obtain 
$$
\widehat{\E}_x[f(S_1(1))-f(x_1)]
\le-\left(p_1(x_1)+o(1)\right)\frac{\mu_1}{2}.
$$
Increasing, if needed, the value of $M$, we conclude that 
$$
\widehat{\E}_x[f(S_1(1))-f(x_1)]\le 0
\quad\text{for all }x\text{ such that }x_1,x_2>M.
$$
Set now
$$
g(u):=\min\{f(u),f(M)\}.
$$
If $x_1>M$ then 
$$
\widehat{\E}_x[g(S_1(1))-g(x_1)]
=\widehat{\E}_x[g(S_1(1))-f(x_1)]
\le \widehat{\E}_x[f(S_1(1))-f(x_1)]\le 0.
$$
Furthermore, for $x_1\le M$ we have 
$$
\widehat{\E}_x[g(S_1(1))-g(x_1)]
=\widehat{\E}_x[g(S_1(1))-f(M)]\le 0.
$$
Thus, this function $g$ satisfies \eqref{eq:sigma-2}.

The relation \eqref{eq:sigma-main} implies that
$$
\lim_{x_1,x_2\to\infty} \frac{W(x)}{V(x_2)}=1.
$$
In particular, $W(x)$ positive for all $x$ with large coordinates, say for $x$ with $|x_1|, |x_2|>M$ with some sufficiently large. 

Using the harmonicity property we can show that $W(x)$ is positive for all $x$ such that $\P(|x_1+S_1(n_0)|>M,x_2+S_2(n_0)|>M,T_x>n_0)>0$ for some $n_0\ge 1$.
Indeed, the harmonicity of $W$ implies that 
\begin{align*}
W(x)&=\E[W(x+S(n_0)),T_x>n_0]\\
&=\E[W(x+S(n_0));|x_1+S_1(n_0)|>M,x_2+S_2(n_0)|>M,T_x>n_0]>0.
\end{align*}
\section{Local limit theorem for the walk conditioned to stay in the positive quadrant}\label{s05}
\subsection{Proof of Theorem~\ref{t11}}
We start by noting that \eqref{eq:llt'} implies that 
$$
\P(x+S(n)=y,T_x>n)=0\quad\text{for all }y\notin D_n(x).
$$

Assume that $y\in D_n(x)$. 
Recalling that $T_x=\min\{\tau_x,\sigma_x\}$ and decomposing the probability 
$\P(x+S(n) =y,T_x>n )$ according to the first crossing of $y$-axis, we obtain
\begin{align}
&
\P(x+S(n) =y,T_x>n )\nonumber\\
&=\P(x+S(n)=y,\tau_x>n)-\P(x+S(n)=y,\tau_x>n,\sigma_x<n)\nonumber\\
&=\P(x+S(n)=y,\tau_x>n)\nonumber\\&\quad -\sum_{k=1}^{n-1}
\sum_{z\in \Z^2\colon z_1<0}
\P(x+S(k)=z,\tau_x>k,\sigma_x=k)
\P(z+S(n-k)=y,\tau_z>n-k).\label{t01}
\end{align}
We now analyse separately different ranges of values $k$ in the sum on the right hand side of \eqref{t01}.
Fix some $N\in\N$ and consider summands with $k\in (N,n/2)\cap\Z$. According to Corollary~\ref{d06} for $k\in 
(N,n/2)\cap \Z$,
\begin{align*}
\P(z+S(n-k)=y,\tau_z >n-k)
&\leq \frac{CV(z_2)H(\min \{\sqrt{n-k},y_2\})}{(n-k)^2}\nonumber\\
&\leq \frac{C_1 V(z_2)}{(n-k)^{3/2}}\leq \frac{C_2V(z_2)}{n^{3/2}}
\end{align*} 
uniformly in $y\in \N^2$. This implies that
\begin{align*}
&\sum_{z\in \Z^2\colon z_1<0}
\P(x+S(k)=z,\tau_x>k,\sigma_x=k)
\P(z+S(n-k)=y,\tau_z>n-k)\nonumber\\
&\hspace{1cm}\leq \frac{C_2}{n^{3/2}}
\sum_{z\in \Z^2\colon z_1<0}
V(z_2)\P(x+S(k)=z,\tau_x>k,\sigma_x=k)\nonumber\\
&\hspace{1cm}=\frac{C_2}{n^{3/2}}\E \!\left[ V(x_2+S_2(\sigma_x)),\tau_x>k,\sigma_x=k\right].
\end{align*}
Consequently,
\begin{align}
&
\sum_{k=N+1}^{n/2}\sum_{z\in \Z^2\colon z_1<0}
\P(x+S(k)=z,\tau_x>k,\sigma_x=k)
\P(z+S(n-k)=y,\tau_z>n-k)\nonumber\\
&\hspace{1cm}\leq 
\frac{C_2}{n^{3/2}}
\E\!\left[V(x_2+S_2(\sigma_x))\1_{\{\tau_x>\sigma_x\}},\sigma_x\in (N,\infty)\right].\label{t02}
\end{align}
Set $M_n:=V(x_2+S_2(n))\1_{\{\tau_x>n\}}$ for $n\in \N_0$.
We know that the sequence $(M_n)_{n\in \N_0}$ is a martingale. Then, by the optional stopping theorem, we have
\begin{align*}
V(x_2)=M_0
=\E[M_{\sigma_x\wedge n}]
=\E [M_{\sigma_x},\sigma_x\leq n]
+\E [M_n,\sigma_x>n].
\end{align*}
Since $M_n\geq 0$ for all $n\in \N_0$ this implies that
\begin{align*}
\E [M_{\sigma_x},\sigma_x\leq n]\leq V(x_2)
\end{align*}
for all $n\in \N_0$.
Letting $n\to\infty$ and using the definition of $(M_n)_{n\in \N_0}$ we then have
\begin{align*}
\E [V(x_2+S_2(\sigma_x))\1_{\{\tau_x>\sigma_x\}},\sigma_x<\infty]\leq V(x_2).
\end{align*}
Combining this with the fact that 
$\P(\sigma_x\in (N,\infty))\to 0$ as $N\to\infty$ we conclude that
\begin{align*}
\varepsilon_N:=
\E [V(x_2+S_2(\sigma_x))\1_{\{\tau_x>\sigma_x\}},\sigma_x\in (N,\infty)]\to 0.
\end{align*} 
This and \eqref{t02} imply that, uniformly in $N$,
\begin{align}\nonumber
&n^{3/2}\sum_{k=N+1}^{n/2}\sum_{z\in \Z^2\colon z_1<0}\P(x+S(k)=z,\tau_x>k,\sigma_x=k)
\P(z+S(n-k)=y,\tau_z>n-k)\\
&\hspace{2cm}
\leq C_2\varepsilon_N.\label{t09}
\end{align}
To estimate the sum over $k> n/2$ we notice that
\begin{align}
&\sum_{k=n/2}^{n-1}\sum_{z\in \Z^2\colon z_1<0}
\P(x+S(k)=z,\tau_x>k,\sigma_x=k)
\P(z+S(n-k)=y,\tau_z>n-k)\nonumber\\
&\hspace{1cm}=\P(x+S(n)=y,\tau_x>n,\sigma_x\in (n/2,n))\nonumber\\
&\hspace{1cm}\leq \P(x+S(n)=y,\tau_x>n,S_1(n/2)<\mu_1 n/4)\nonumber\\
&\hspace{1cm}\quad 
+\P\!\left(\tau_x>n/2,\min_{k\in (n/2,n)}  [S_1(k)-S_1(n/2)]<-\mu_1 n/4) \right).\label{t07}
\end{align}
For the second probability term we have, due to Lemma~\ref{lem:loc.bound.1},
\begin{align}
&
\P\!\left(\tau_x>n/2,\min_{k\in (n/2,n)}  [(S_1(k)-S_1(n/2)]<-\mu_1 n/4) \right)\nonumber\\
&\hspace{1cm}\leq \P(\tau_x\geq n/2)\P\!\left(\min_{k\leq n}S_1(k)\leq -\mu_1 n/4\right)
\nonumber\\
&\hspace{1cm}\leq \frac{CV(x_2)}{\sqrt{n}}\P\!\left(\min _{k\ge 1}
S_1(k)\leq -\mu_1 n/4
\right).
\label{t03}
\end{align}
According to Theorem~2 in \cite{Kiefer_Wolfowitz1956}, the assumption 
$\E X_1^2<\infty$ implies that
$\E[-\min _{k\ge 1}S_1(k)]<\infty$. This implies that 
\begin{align}
\P\!\left(\min _{k\ge 1}
S_1(k)\leq -\mu_1 n/4
\right)=o\left(\frac{1}{n}\right).
\end{align}
This and \eqref{t03} imply that
\begin{align}
\P\!\left(\tau_x>n/2,\min_{k\in (n/2,n)}  [(S_1(k)-S_1(n/2)]<-\mu_1 n/4
\right)=o(n^{-3/2}).\label{t06}
\end{align}
By the Markov property at time $n/2$,
\begin{align*}
&\P(x+S(n)=y,\tau_x>n,x_1+S_1(n)\leq \mu_1 n/4)\nonumber\\
&\leq \sum_{z\in \Z^2\colon z_1<\mu_1 n/4}
\P(x+S(n/2)  =z,\tau_x>n/2 )\P(z+S(n/2)=y).
\end{align*}
The fact that $\P(z+S(n/2) =y )\leq C/n$
uniformly in $z,y$ and Proposition~\ref{t05} prove that
\begin{align}
&
\P(x+S(n)=y,\tau_x>n,x_1+S_1(n)\leq \mu_1 n/4)\nonumber\\
&
\leq \frac{C}{n}\P(x_1+S_1(n/2) \leq \mu_1 n/4,\tau_x> n/2)= o(n^{-3/2}).\label{t04}
\end{align}
Plugging \eqref{t04} and \eqref{t06} into \eqref{t07} we have
\begin{align}\label{t10}
&\nonumber
\sum_{k=n/2}^{n-1}\sum_{z\in \Z^2\colon z_1<0}
\P(x+S(k)=z,\tau_x>k,\sigma_x=k)
\P(z+S(n-k)=y,\tau_z>n-k)\\
&\hspace{2cm}=o(n^{-3/2})
\end{align} 
uniformly in $y$. Thus, it remains to consider the case $k\leq N$. According to Proposition~\ref{d07} we have
\begin{align*}
n^{3/2}\P(z+S(n-k)=y,\tau_z>n-k)-p\!\left(\frac{y-n\mu}{\sqrt{n}}\right)V(z_2) \to 0
\end{align*}
uniformly in $y\in\overline{D}_{n-k}(z)$ for every fixed $z$. Noting now that Corollary~\ref{d06} allows one to use Lebesgue's theorem,
\begin{align*}
&
n^{3/2}\sum_{z\in \Z^2\colon z_1<0}\P(x+S(k)=z,\tau_x>k,\sigma_x=k)
\P(z+S(n-k)=y,\tau_z>n-k)\nonumber\\
&\quad 
-p\!\left(\frac{y-n\mu}{\sqrt{n}}\right)
\E \!\left[V(x_2+S_2(\sigma_x)),\tau_x>k,\sigma_x=k\right]\to 0
\end{align*} 
uniformly in $y\in\overline{D}_n(x)$. Summing over all $k\leq N$ we get
\begin{align}
&
n^{3/2}\sum_{k=1}^{N}\sum_{z\in \Z^2\colon z_1<0}\P(x+S(k)=z,\tau_x>k,\sigma_x=k)
\P(z+S(n-k)=y,\tau_z>n-k)\nonumber\\
&\quad 
-p\!\left(\frac{y-n\mu}{\sqrt{n}}\right)
\E \!\left[V(x_2+S_2(\sigma_x)),\tau_x>k,\sigma_x\leq N\right]\to 0.\label{t08}
\end{align} 
Combining
\eqref{t01}, \eqref{t09}, \eqref{t10}, and \eqref{t08},
letting $N\to\infty$ and recalling that
$$
\lim_{N\to\infty}
\E \!\left[V(x_2+S_2(\sigma_x)),\tau_x>k,\sigma_x\leq N\right]
\to W(x)
$$
we complete the proof of Theorem~\ref{t11}. 
\subsection{Proof of Theorem~\ref{t12}}
By the Markov property at time $n/2$ we have
\begin{align*}
&
\P(x+S(n)=y, T_x>n)\nonumber\\
&=\sum_{z\in D_{n/2}(x)}
\P(x+S(n/2)=z,T_x>n/2)\nonumber\\
&\qquad \qquad\qquad
\P(z+S(n/2)=y,z+S(k)\in \N^2\,
\text{ for all }\,k\leq n/2
).
\end{align*}
Inverting the time in the second half of the trajectory we have
\begin{align*}
&
\P(x+S(n)=y, T_x>n)\nonumber\\
&=\sum_{z\in D_{n/2}(x)}
\P(x+S(n/2)=z,T_x>n/2)\P(y+S'(n/2)=z,T'_y>n/2),
\end{align*}
where $S'=-S$ and $T'_y=\inf\{k\in \N\colon y+S'(k)\notin \N^2\}$. Applying Theorem~\ref{t11}, we have
\begin{align*}
&\P(x+S(n)=y,T_x>n)\nonumber\\
&\hspace{1cm}\nonumber
= d_1d_2\varkappa\sum_{z\in D_{n/2}(x)}
\frac{W(x)}{(n/2)^{3/2}}
p\!\left(\frac{z-n\mu/2}{\sqrt{n/2}}\right)
\P(y+ S'(n/2)=z,T_y'>n/2)\\
&\hspace{2cm}+ o \!\left(n^{-3/2} \P(\tau'_y>n/2) \right).
\end{align*}
Recalling that $\P(\tau'_y>n/2) \leq \frac{CH(y_2)}{\sqrt{n}}$
we conclude that, uniformly in $y$,
\begin{align*}
&\P(x+S(n)=y,T_x>n)\nonumber\\
&
= d_1d_2\varkappa\sum_{z\in D_{n/2}(x)}
\frac{W(x)}{(n/2)^{3/2}}
p\!\left(\frac{z-n\mu/2}{\sqrt{n/2}}\right)
\P(y+ S'(n/2)=z,T_y'>n/2)+o\!\left(\frac{H(y_2)}{n^2}\right).
\end{align*}
We next notice that
\begin{align*}
& \sum_{z\in D_{n/2}(x)}
\frac{W(x)}{(n/2)^{3/2}}
p\!\left(\frac{z-n\mu/2}{\sqrt{n/2}}\right)
\P(y+ S'(n/2)=z,T_y'>n/2)\nonumber\\
&\hspace{1cm}
=
 \sum_{z\in \overline{D}_{n/2}(x)}
\frac{W(x)}{(n/2)^{3/2}}
p\!\left(\frac{z-n\mu/2}{\sqrt{n/2}}\right)
\P(y+ S'(n/2)=z,\tau_y'>n/2)
\nonumber\\
&\hspace{1cm}\quad +O(n^{-3/2}\P(\tau'_y>n/2,T'_y\leq n/2) ).
\end{align*}
For all $y$ with $y_1>3\mu n/4$
 we have that
\begin{align*}
\P\!\left(T'_y\leq n/2,\tau'_y>n/2\right)\leq \P\!\left(\min_{k\leq n/2}
S_1'(k)\leq -3\mu_1 n/4 \right)=o(1/n).
\end{align*}
As a result, uniformly in $y$ such that $y_1>\frac{3}{4}\mu_1 n$ and $y_2=o(\sqrt{n})$,
\begin{align*}
&
\P(x+S(n)=y, T_x>n)\nonumber\\
&=2^{3/2}d_1d_2\varkappa\sum_{z\in \N^2}\frac{W(x)}{n^{3/2}}p\!\left(\frac{z-n\mu/2}{\sqrt{n/2}}\right)\P(y+S'(n/2)=z,\tau'_y>n/2)
+o\!\left(\frac{y_2W(x)}{n^2}\right)
\nonumber\\
&=2d_1d_2\varkappa\varkappa'\frac{H(y_2)W(x)}{n^{2}}
\sum_{z\in \N^2}p\!\left(\frac{z-n\mu/2}{\sqrt{n/2}}\right)\P\!\left(y+S'(n/2)=z\middle|\tau'_y>n/2\right)
+o\!\left(\frac{H(y_2)W(x)}{n^2}\right)\nonumber\\
&=2d_1d_2\varkappa\varkappa'\frac{H(y_2)W(x)}{n^2}
\E\! \left[p\!\left(\frac{y+S'(n/2)-n\mu /2}{\sqrt{n/2}}\right)\middle|\tau'_y>n/2\right]
+o\!\left(\frac{H(y_2)W(x)}{n^2}\right).
\end{align*}  
Applying Proposition~\ref{t05} to the walk $S'$, we get the desired relation
with the function
$$
q(z_1)=2\varkappa\varkappa'\overline{q}(z_1,z_2),
$$
where 
\begin{align*}
    \bar{q}(y)=\int_{\R\times [0,\infty)}p(y+\tilde{y})p(\tilde{y})\,d\tilde{y},
    \quad y\in \R^2.
\end{align*}
Recall that $p\colon \mathbb{R}^2\to \mathbb{R}$ satisfy that
\begin{align*}
p(y)=
     \frac{y_2}{\sigma_2 \sqrt{2\pi(\sigma_1^2\sigma_2^2-\rho^2)}} \exp\!\left(-\frac{\sigma_2^2y_1^2+\sigma_1^2y_2^2-2\rho y_1y_2 }{2(\sigma_1^2\sigma_2^2-\rho^2)}\right)
\end{align*}
for all $y\in \mathbb{R}^2$. Let us calculate $\bar{q}(y_1,0)$ for $y_1\in \mathbb{R}$. We have
\begin{align*}
    \bar{q}(y_1, 0) = \int_{-\infty}^{\infty} \int_{0}^{\infty} p(y_1 + \tilde{y}_1, \tilde{y}_2) p(\tilde{y}_1, \tilde{y}_2) \, d\tilde{y}_2 \, d\tilde{y}_1
\end{align*}
for  $y_1\in \mathbb{R}$. 
To shorten the notation, let $D = \sigma_1^2\sigma_2^2 - \rho^2$.
Then
$$p(y_1 + \tilde{y}_1, \tilde{y}_2) = \frac{\tilde{y}_2}{\sigma_2 \sqrt{2\pi D}} \exp\left( -\frac{\sigma_2^2 (y_1 + \tilde{y}_1)^2 + \sigma_1^2 \tilde{y}_2^2 - 2\rho (y_1 + \tilde{y}_1)\tilde{y}_2}{2D} \right)$$ and
$$p(\tilde{y}_1, \tilde{y}_2) = \frac{\tilde{y}_2}{\sigma_2 \sqrt{2\pi D}} \exp\left( -\frac{\sigma_2^2 \tilde{y}_1^2 + \sigma_1^2 \tilde{y}_2^2 - 2\rho \tilde{y}_1 \tilde{y}_2}{2D} \right).$$
Since \begin{align*}
- \frac{(\sigma_2^2 y_1 - 2\rho \tilde{y}_2)^2}{2\sigma_2^2} = -\frac{\sigma_2^4 y_1^2 - 4\rho \sigma_2^2 y_1 \tilde{y}_2 + 4\rho^2 \tilde{y}_2^2}{2\sigma_2^2} = -\frac{\sigma_2^2 y_1^2}{2} + 2\rho y_1 \tilde{y}_2 - \frac{2\rho^2 \tilde{y}_2^2}{\sigma_2^2}
\end{align*} we have
\begin{align*}
&
\left[ \sigma_2^2 (y_1 + \tilde{y}_1)^2 + \sigma_1^2 \tilde{y}_2^2 - 2\rho (y_1 + \tilde{y}_1)\tilde{y}_2 \right] + \left[ \sigma_2^2 \tilde{y}_1^2 + \sigma_1^2 \tilde{y}_2^2 - 2\rho \tilde{y}_1 \tilde{y}_2 \right]\\
& = \sigma_2^2 y_1^2 + 2\sigma_2^2 y_1 \tilde{y}_1 + \sigma_2^2 \tilde{y}_1^2 + \sigma_1^2 \tilde{y}_2^2 - 2\rho y_1 \tilde{y}_2 - 2\rho \tilde{y}_1 \tilde{y}_2 + \sigma_2^2 \tilde{y}_1^2 + \sigma_1^2 \tilde{y}_2^2 - 2\rho \tilde{y}_1 \tilde{y}_2\\
&=
2\sigma_2^2 \tilde{y}_1^2 + 2\tilde{y}_1 (\sigma_2^2 y_1 - 2\rho \tilde{y}_2) + (\sigma_2^2 y_1^2 + 2\sigma_1^2 \tilde{y}_2^2 - 2\rho y_1 \tilde{y}_2)\\
&=
2\sigma_2^2 \tilde{y}_1^2 + 2\tilde{y}_1 (\sigma_2^2 y_1 - 2\rho \tilde{y}_2)+(\sigma_2^2 y_1^2 + 2\sigma_1^2 \tilde{y}_2^2 - 2\rho y_1 \tilde{y}_2)
\\&=2\sigma_2^2 \left[ \tilde{y}_1^2 + 2\tilde{y}_1 \left( \frac{\sigma_2^2 y_1 - 2\rho \tilde{y}_2}{2\sigma_2^2} \right) \right]+(\sigma_2^2 y_1^2 + 2\sigma_1^2 \tilde{y}_2^2 - 2\rho y_1 \tilde{y}_2)
&\\
&
=
2\sigma_2^2 \left[ \tilde{y}_1 + \frac{\sigma_2^2 y_1 - 2\rho \tilde{y}_2}{2\sigma_2^2} \right]^2 - 2\sigma_2^2 \left( \frac{\sigma_2^2 y_1 - 2\rho \tilde{y}_2}{2\sigma_2^2} \right)^2
+(\sigma_2^2 y_1^2 + 2\sigma_1^2 \tilde{y}_2^2 - 2\rho y_1 \tilde{y}_2)
\\
&=
2\sigma_2^2 \left[ \tilde{y}_1 + \frac{\sigma_2^2 y_1 - 2\rho \tilde{y}_2}{2\sigma_2^2} \right]^2  -\frac{\sigma_2^2 y_1^2}{2} + 2\rho y_1 \tilde{y}_2 - \frac{2\rho^2 \tilde{y}_2^2}{\sigma_2^2}
+(\sigma_2^2 y_1^2 + 2\sigma_1^2 \tilde{y}_2^2 - 2\rho y_1 \tilde{y}_2)\\
& = 2\sigma_2^2 \left[ \tilde{y}_1 + \frac{\sigma_2^2 y_1 - 2\rho \tilde{y}_2}{2\sigma_2^2} \right]^2 + \frac{\sigma_2^2 y_1^2}{2} + \frac{2(\sigma_1^2\sigma_2^2-\rho^2) \tilde{y}_2^2}{\sigma_2^2}\\
& = 2\sigma_2^2 \left[ \tilde{y}_1 + \frac{\sigma_2^2 y_1 - 2\rho \tilde{y}_2}{2\sigma_2^2} \right]^2 + \frac{\sigma_2^2 y_1^2}{2} + \frac{2D \tilde{y}_2^2}{\sigma_2^2}.
\end{align*}
Thus,
\begin{align*}
    p(y_1 + \tilde{y}_1, \tilde{y}_2) p(\tilde{y}_1, \tilde{y}_2) = \frac{\tilde{y}_2^2}{2\pi D\sigma_2^2}\exp\left(-\frac{2\sigma_2^2 \left[ \tilde{y}_1 + \frac{\sigma_2^2 y_1 - 2\rho \tilde{y}_2}{2\sigma_2^2} \right]^2 + \frac{\sigma_2^2 y_1^2}{2} + \frac{2D \tilde{y}_2^2}{\sigma_2^2}}{2D}\right)
\end{align*}
and
$$\bar{q}(y_1, 0) = \frac{1}{2\pi \sigma_2^2 D} \int_{0}^{\infty} \tilde{y}_2^2 \exp\left( -\frac{\sigma_2^2 y_1^2/2 + 2D \tilde{y}_2^2 / \sigma_2^2}{2D} \right) \left[ \int_{-\infty}^{\infty} \exp\left( -\frac{2\sigma_2^2 (\tilde{y}_1 + \frac{\sigma_2^2 y_1 - 2\rho \tilde{y}_2}{2\sigma_2^2})^2}{2D} \right) d\tilde{y}_1 \right] d\tilde{y}_2$$
Noting that
\begin{align*}
     \int_{-\infty}^{\infty} \exp\left( -\frac{\sigma_2^2 (\tilde{y}_1 + \frac{\sigma_2^2 y_1 - 2\rho \tilde{y}_2}{2\sigma_2^2})^2}{D} \right) d\tilde{y}_1 =\sqrt{\frac{\pi D}{\sigma_2^2}}
\end{align*}
we have that
\begin{align*}
\bar{q}(y_1, 0) &= \frac{1}{2\pi \sigma_2^2 D} \int_{0}^{\infty} \tilde{y}_2^2 \exp\left( -\frac{\sigma_2^2 y_1^2/2 + 2D \tilde{y}_2^2 / \sigma_2^2}{2D} \right)\sqrt{\frac{\pi D}{\sigma_2^2}} d\tilde{y}_2\\& = \frac{\sqrt{\pi D}/\sigma_2}{2\pi \sigma_2^2 D} \exp\left( -\frac{\sigma_2^2 y_1^2}{4D} \right) \int_{0}^{\infty} \tilde{y}_2^2 \exp\left( -\frac{\tilde{y}_2^2}{\sigma_2^2} \right) d\tilde{y}_2
\end{align*}
where
\begin{align*}
   \int_{0}^{\infty} \tilde{y}_2^2 \exp\left( -\frac{\tilde{y}_2^2}{\sigma_2^2} \right) d\tilde{y}_2 = \frac{\sigma_2^3 \sqrt{\pi}}{4}.
\end{align*}
Hence,
$$\bar{q}(y_1, 0) = \left( \frac{\sqrt{\pi D}}{2\pi \sigma_2^3 D} \right) \cdot \left( \frac{\sigma_2^3 \sqrt{\pi}}{4} \right) \cdot \exp\left( -\frac{\sigma_2^2 y_1^2}{4D} \right).$$
Recalling that $D = \sigma_1^2\sigma_2^2 - \rho^2$, we have 
\begin{equation}\label{eq:q-def}
\bar{q}(y_1, 0) = \frac{1}{8\sqrt{\sigma_1^2\sigma_2^2 - \rho^2}} \exp\left( -\frac{\sigma_2^2 y_1^2}{4(\sigma_1^2\sigma_2^2 - \rho^2)} \right).
\end{equation}
\subsection{Proof of Corollary~\ref{t15}}

We start by noting that 
\begin{align}
    \P(x_2+S_2(n)=y_2,T_x>n)
    &=\P(x_2+S_2(n)=y_2,\lvert S_1(n)-n\mu\rvert\leq A\sqrt{n},T_x>n)\nonumber\\
    &\quad +
    \P(x_2+S_2(n)=y_2,\lvert S_1(n)-n\mu\rvert>A\sqrt{n},T_x>n)\label{t14}
\end{align}
where $A\geq 1$ is any fixed number. Applying Theorem~\ref{t12} we have
\begin{align}
    &\P(x_2+S_2(n)=y_2,\lvert S_1(n)-n\mu\rvert\leq A\sqrt{n},T_x>n)\nonumber\\
    &\sim d_1d_2\frac{H(y_2)W(x)}{n^2}\sum_{y_1\in \Z\colon \lvert y_1-x_1-n\mu\rvert\leq A\sqrt{n}} q\!\left(\frac{y_1-n\mu}{\sqrt{n}}\right)\nonumber\\
    &\sim d_1d_2\frac{H(y_2)W(x)}{n^{3/2}}\int_{-A}^{A} q(z)\,dz.
    \label{t13}
\end{align}
We split the second probability term on the r.h.s. of \eqref{t14} into two parts
\begin{align*}
    &\P(x_2+S_2(n)=y_2,\lvert S_1(n)-n\mu\rvert>A\sqrt{n},T_x>n)\nonumber\\
    &\leq 
    \P\!\left(x_2+S_2(n)=y_2,\left\lvert S_1(\frac{n}{2})-\frac{n\mu}{2}\right\rvert>\frac{A\sqrt{n}}{2},\tau_x>n\right)\nonumber\\
    &\quad +\P\!\left(x_2+S_2(n)=y_2,\left\lvert S_1(n)-S_1(\frac{n}{2})-\frac{n\mu}{2}\right\rvert>\frac{A\sqrt{n}}{2},\tau_x>n\right).   
\end{align*}
By the Markov property at time $n/2$,
\begin{align*}
    &\P\!\left(x_2+S_2(n)=y_2,\left\lvert S_1(\frac{n}{2})-\frac{n\mu}{2}\right\rvert>\frac{A\sqrt{n}}{2},\tau_x>n\right)\nonumber\\
    &=\sum_{z_2=1}^\infty
    \P\!\left(x_2+S_2(\frac{n}{2})=z_2,\tau_x\geq \frac{n}{2},\left\lvert S_1(\frac{n}{2})-\frac{n\mu}{2}\right\rvert>\frac{A\sqrt{n}}{2}\right)
    \P\!\left(y_2+S'(\frac{n}{2})=z_2,\tau'_y> \frac{n}{2}\right)
    \nonumber\\
     &\leq C\frac{H(y_2)}{n}\P\!\left(
    \left\lvert S_1(\frac{n}{2})-\frac{n\mu}{2}\right\rvert>\frac{A\sqrt{n}}{2},\tau_x>\frac{n}{2}
    \right)\nonumber\\
     &\leq C\frac{V(x_2)H(y_2)}{n^{3/2}}\P\!\left(
    \left\lvert S_1(\frac{n}{2})-\frac{n\mu}{2}\right\rvert>\frac{A\sqrt{n}}{2}\middle|\tau_x>\frac{n}{2}
    \right).
\end{align*}
It follows from Theorem~\ref{t05} that
\begin{align*}
\limsup_{n\to\infty}
\P\!\left(
    \left\lvert S_1(\frac{n}{2})-\frac{n\mu}{2}\right\rvert>\frac{A\sqrt{n}}{2}\middle|\tau_x>\frac{n}{2}
    \right)\leq \delta(A)
\end{align*}
with some $\delta(A)\to 0$ as $A\to \infty$. As a result,
\begin{align*}
&\P\!\left(x_2+S_2(n)=y_2,\left\lvert S_1(\frac{n}{2})-\frac{n\mu}{2}\right\rvert>\frac{A\sqrt{n}}{2},\tau_x>n\right)\nonumber\\
&\leq C\delta(A)\frac{V(x_2)H(y_2)}{n^{3/2}}.
\end{align*}
By the symmetry,
\begin{align*}
    &\P\!\left(x_2+S_2(n)=y_2,\left\lvert S_1(n)-S_1(\frac{n}{2})-\frac{n\mu}{2}\right\rvert>\frac{A\sqrt{n}}{2},\tau_x>n\right)\nonumber\\
    &\leq C\delta(A)\frac{V(x_2)H(y_2)}{n^{3/2}}.
\end{align*}
Combining these two bounds we obtain
\begin{align*}
    &\P(x_2+S_2(n)=y_2,\lvert S_1(n)-n\mu\rvert>A\sqrt{n},T_x>n)\nonumber\\
    &\leq C\delta(A)\frac{V(x_2)H(y_2)}{n^{3/2}}.
\end{align*} This and \eqref{t13} complete the proof of Corollary~\ref{t15}.

\subsection{Singular walks ending at one of the axes}
In this paragraph we show how to apply our probabilistic results to the enumeration problems of lattice paths in the positive quadrant with possible steps
$(-1,1), (1,1)$ and $(1,-1)$. Let $N_n(x,y)$ denote the number of paths of length $n$ which start at $x$, end at $y$ and do not leave the positive quadrant $\N^2$.
Then we have
$$
N_n(x,y)=3^n\P(x+S(n)=y,T_x>n),
$$
where $S(n)$ is a random walk with i.i.d. increments which have the uniform distribution on the set of possible steps: 
$$
\P( X(1)=(1,1))=\P( X(1)=(-1,1))= \P( X(1)=(1,-1))=\frac{1}{3}.
$$
This random walk has positive drift: $\E[X_1(1)]=\E[X_2(1)]=\frac{1}{3}$. In order to apply our results and to determine the asymptotic behaviour of the probability$\P ((1,1)+S(n)=(m,1), T_x>n)$ in the case when $n\to\infty$ and 
$\frac{m}{n}\to\mu_1\in(0,1)$, we perform an exponential change of measure.
Set 
\begin{align*}
\varphi(h)= \E[ e^{h_1X_1(1)+h_2X_2(1)}]=
\frac{1}{3}\left(e^{h_1-h_2}+e^{h_1+h_2}+e^{-h_1+h_2}\right),
\quad h=(h_1,h_2)\in \R^2.
\end{align*} 
Let $\P^{(h)}$ denote the probability measure under which $X(1),X(2),\ldots$ are i.i.d. with the common distribution
\begin{align*}
    &\P^{(h)}(X(1)=(1,-1))=\frac{e^{h_1-h_2}}{3\varphi(h)},\\ 
    &\P^{(h)}(X(1)=(1,1))=\frac{e^{h_1+h_2}}{3\varphi(h)},\\ 
    &\P^{(h)}(X(1)=(-1,1))=\frac{e^{-h_1+h_2}}{3\varphi(h)},
\end{align*}
and $\E^{(h)}$  be the corresponding expectation. To apply Theorem~\ref{t12} with $y_2=1$ we have to  find $h=(h_1,h_2)$ such that 
$\E ^{(h)}[X_2(1)]=0$ and $\E^{(h)}[X_1(1)]=\mu_1$. 
From $\E ^{(h)}[X_2(1)]=0$ we get
\begin{align*}
0=\P^{(h)}(X_2(1)=1)-\P^{(h)}(X_2(1)=-1)= \frac{e^{h_1+h_2}+e^{-h_1+h_2}-e^{h_1-h_2}}{3\varphi(h)}.
\end{align*}
From
$\E^{(h)} [X_1(1)]=\mu_1$ we obtain
\begin{align*}
 \mu_1=\P^{(h)} (X_1(1)=1)-\P^{(h)} (X_1(1)=-1)= \frac{
e^{h_1+h_2}+e^{h_1-h_2}-e^{h_2-h_1}}{3\varphi(h)}
\end{align*}
and hence $$e^{h_1+h_2}+e^{h_1-h_2}-e^{h_2-h_1}=3{\mu}_1\varphi(h).$$
To summarize we have a system of two equations:
\begin{align*}
&    e^{h_1+h_2}+e^{-h_1+h_2}-e^{h_1-h_2}=0,\\
&
    e^{h_1+h_2}+e^{h_1-h_2}-e^{h_2-h_1}={\mu}_1\left(e^{h_1-h_2}+e^{h_1+h_2}+e^{-h_1+h_2}\right).
\end{align*}
Letting $x=e^{h_1}>0$, $y=e^{h_2}>0$ we obtain
\begin{align}\label{c01a}
    &xy+\frac{y}{x}-\frac{x}{y}=0,
\\
&xy+\frac{x}{y}-\frac{y}{x}= {\mu}_1 \left( xy+\frac{x}{y}+\frac{y}{x}\right).
    \label{c02a}
\end{align}Multiplying \eqref{c01a} with $xy$ implies that
$x^2y^2+y^2-x^2=0$, i.e., $$y^2=\frac{x^2}{x^2+1}.$$ Multiplying \eqref{c01a} with $xy$ implies
that $x^2y^2+x^2-y^2={\mu}_1(x^2y^2+x^2+y^2)$, i.e.,
$
(x^2-1)y^2+x^2= {\mu}_1((x^2+1)y^2+x^2).
$
Plugging $y^2=\frac{x^2}{x^2+1}$ in this equation we get
that
$\frac{(x^2-1)x^2}{x^2+1}+x^2= 2x^2{\mu}_1 $. Hence,
$\frac{x^2-1}{x^2+1}+1=2{\mu}_1$. Thus,
$\frac{x^2}{x^2+1}={\mu}_1$. Therefore,
$x^2={\mu}_1x^2+{\mu}_1$ and we get
$x^2=\frac{{\mu}_1}{1-{\mu}_1}$. This and the fact that $y^2=\frac{x^2}{x^2+1}$ imply that
$$y^2=\frac{\frac{{\mu}_1}{1-{\mu}_1}}{\frac{{\mu}_1}{1-{\mu}_1}+1}={\mu}_1.$$
We conclude that
$h_1=\log x=\frac{1}{2}\log(\frac{{\mu}_1}{1-{\mu}_1}) $, $h_2=\log y=\frac{1}{2}\log {\mu}_1$,  
\begin{align}
    3\varphi(h)=xy+\frac{x}{y}+\frac{y}{x}=\frac{{\mu}_1}{\sqrt{1-{\mu}_1}}+\frac{1}{\sqrt{1-{\mu}_1}}+\sqrt{1-{\mu}_1}=\frac{2}{\sqrt{1-{\mu}_1}},\label{c06}
\end{align}
and\begin{align*}
    e^{-h_1(m-1)}= \frac{1}{x^{m-1}}=\left(\frac{1-{\mu}_1}{{\mu}_1}\right)^{\frac{m-1}{2}}.
\end{align*}
Next, assuming that $m-\mu_1 n=o(\sqrt{n})$ and applying Theorem \ref{t12}, we obtain
\begin{align}
    \P^{(h)}\!\left((1,1)+S(n)=(m,1),T_{(1,1)}>n\right)= \frac{c}{n^2}(1+o(1))\label{c07}
\end{align} where $c$ is a constant depending only on $h$, i.e., on $\mu_1$.
Next,
\begin{align*}
    \P \left((1,1)+S(n)=(m,1),T_{(1,1)}>n\right)=(\varphi(h))^n
    e^{-h_1(m-1)}\P^{(h)}\!\left((1,1)+S(n)=(m,1),T_{(1,1)}>n\right).
\end{align*}
This and 
\eqref{c06}--\eqref{c07}
prove that
\begin{align*}
\P \left((1,1)+S(n)=(m,1),T_{(1,1)}>n\right)= \left(\frac{2}{3\sqrt{1-{\mu}_1}}\right)^n\left(\frac{1-{\mu}_1}{{\mu}_1}\right)^{\frac{m-1}{2}}\frac{c}{n^2}(1+o(1)).
\end{align*}    

We now derive asymptotics for the number $M_n(x)$ of walk of length $n$ which start at a point $x$ and end at the line $\{(m,1),\ m\ge1\}$. By the definition of the measure $\P^{(h)}$,
\begin{align*}
 M_n(x)
 &=3^n\P(x_2+S_2(n)=1,T_x>n)\\
 &=3^n\sum_{m=1}^\infty\P(x+S(n)=(m,1),T_x>n)\\
 &=3^n\sum_{m=1}^\infty (\varphi(h))^ne^{-h_1(m-x_1)}e^{-h_2(1-x_2)}
        \P^{(h)}(x+S(n)=(m,1),T_x>n).
\end{align*}
In the particular case $h_1=0$ we then have 
\begin{align}
\label{eq:M-repr}
M_n(x)=(3\varphi(h))^ne^{h_2(x_2-1)}\P^{(h)}(x_2+S_2(n)=1,T_x>n)
\quad\text{for }h=(0,h_2).
\end{align}
In order to apply Corollary~\ref{t15} we have to choose $h_2$ so that
$\E^{(h)}[X_2(1)]=0$. The corresponding to this condition equation \eqref{c01a}
in the case $h_1=0(x=1)$ reduces to $2y-y^{-1}=0$. Thus, $y=2^{-1/2}$ and, consequently, $h_2=-\frac{1}{2}\log2$. Set, for brevity,
$h^*=(0,\frac{1}{2}\log 2)$. Then, using \eqref{eq:M-repr} with $h=h^*$ and noting that $\varphi(h^*)=\frac{2\sqrt{2}}{3}=\frac{2^{3/2}}{3}$, we obtain
$$
M_n(x)=2^{3n/2}2^{(1-x_2)/2}\P^{(h^*)}(x_2+S_2(n)=1,T_x>n).
$$
Applying now Corollary~\ref{t15}, we infer that there exists a constant $C_0$ such that 
$$
M_n(x)\sim C_0 W^*(x)2^{(1-x_2)/2}\frac{2^{3n/2}}{n^{3/2}}
\quad\text{as }n\to\infty,
$$
where
$$
W^*(x)=x_2-\E^{(h^*)}[x_2+S_2(\sigma_x);\tau_x>\sigma_x;\sigma_x<\infty].
$$
Here we have also used the fact that the renewal functions $V(x)$ and $H(x)$ are linear for the simple symmetric random walk $S_2(n)$.

{
\bibliographystyle{acm}
\bibliography{literature}
}

\end{document}